\documentclass[a4paper,twoside, leqno]{article}
\usepackage{latexsym,bm}
\usepackage{mathrsfs,amsmath,amssymb,amsthm}
\usepackage[dvipdfm]{graphicx}
\usepackage{amscd}
\usepackage{jmsj2009}
\usepackage[all]{xy}
\newtheorem{thm}{Theorem}[section]
\newtheorem{dfn}[thm]{Definition}
\newtheorem{prp}[thm]{Proposition}%[section]k
\newtheorem{conj}[thm]{Conjecture}%[section
\newtheorem{lmm}[thm]{Lemma}%[section]
\newtheorem{crl}[thm]{Corollary}%[section]
\newtheorem{exa}[thm]{Example}%[section]
\newtheorem{rmk}[thm]{Remark}%[section]
\numberwithin{equation}{section}
\theoremstyle{remark}
\usepackage[all]{xy}

\newcommand{\Hol}{\mbox{{\rm Hol}}}
\newcommand{\Alg}{\mbox{{\rm Alg}}}

\newcommand{\Z}{\Bbb Z}
\newcommand{\C}{\Bbb C}
\newcommand{\R}{\Bbb R}

\newcommand{\T}{\Bbb T}
\renewcommand{\P}{{\rm P}}

\newcommand{\RP}{\Bbb R\mbox{{\rm P}}}

\newcommand{\Map}{\mbox{{\rm Map}}}

\newcommand{\CP}{\Bbb C {\rm P}}

\newcommand{\dis}{\displaystyle}
\newcommand{\p}{\prime}

\newcommand{\E}{\tilde{E}}
\newcommand{\XD}{X^{\Delta}}
\newcommand{\SZ}{{\mathcal{X}}^{D}}
\newcommand{\SZd}{{\mathcal{X}}^{D+\textbf{\textit{a}}}}
\newcommand{\I}{\mbox{{\rm (i)}}}
\newcommand{\II}{\mbox{{\rm (ii)}}}
\newcommand{\III}{\mbox{{\rm (iii)}}}
\newcommand{\IV}{\mbox{{\rm (iv)}}}

\newcommand{\XS}{X_{\Sigma}}
\newcommand{\ZS}{Z_{\Sigma}}
\newcommand{\GS}{G_{\Sigma}}
\newcommand{\n}{{\bf n}}
\newcommand{\db}{d_{\rm min}}
\newcommand{\rmin}{r_{\rm min}}
\begin{document}

\title{Spaces of algebraic maps from real projective spaces to toric varieties}

\author{Andrzej  \textsc{Kozlowski},
Masahiro \textsc{Ohno} and
Kohhei \textsc{Yamaguchi}\footnote{%
The second author and third  author
were supported by 
JSPS KAKENHI Grant Number
22540043, 23540079 and 26400083.}}



\label{startpage}

\maketitle

%%%(Abstract)%%%%%%%%%%%%
\abstract{%************
The problem of approximating the  infinite dimensional space of all continuous maps from an algebraic variety $X$ to an
 algebraic variety  $Y$ by finite dimensional spaces of 
algebraic maps  arises in several areas of geometry and mathematical physics.
An often considered formulation of the problem (sometimes called the Atiyah-Jones problem after 
\cite{AJ}) is to determine a (preferably optimal) integer $n_D$ such that the inclusion from this finite dimensional algebraic space
into the corresponding infinite dimensional one induces isomorphisms of homology
(or homotopy) groups through dimension $n_D$, where $D$ denotes a tuple of integers called the \lq\lq degree\rq\rq
of the algebraic maps and $n_D\to\infty$ as $D\to\infty$.
In this paper we  investigate this problem in the case when
$X$ is a 
real projective space  and $Y$ is a  
smooth compact %complex
toric variety.
}
%%%%%%%(End of Abstract)%%%%

%%%(SECTION 0)%%%%%%%%%%%%%%
\section{Introduction.}\label{section 0}
%%%%%%%%%%%%%%%%%%%%%%%%%%%%

Let $X$ and $Y$ be manifolds with some additional structure ${\cal S}$, e.g. holomorphic, symplectic, real algebraic etc. 
%%%
Let ${\cal S}(X,Y)$ denote the space of base-point preserving continuous maps
$f:X\to Y$ 
preserving the structure ${\cal S}$
and let $\Map^* (X,Y)$ be the space of corresponding continuous maps.
%%%
The relation between the topology of the spaces ${\cal S}(X,Y)$
and  
$\Map^* (X,Y)$ 
has long been an object of study in several areas of topology and geometry (e.g. \cite{BHM}, \cite{CJS}, \cite{Gu2}, %\cite{GKY1}, %\cite{Grom}, 
\cite{GKY2}, \cite{KY1}, \cite{KY3}, \cite{KY4}, \cite{KY5}, \cite{Mo2}, 
\cite{Mo3}, \cite{Se}). 
%%%%%%
In particular, in \cite{Mo2} and \cite{Mo3}
J. Mostovoy considered the case where the structure ${\cal S}$ is that of a complex manifold, and determined an integer $n_D$ such that
the inclusion map
$j_D:\Hol_D^*(X,Y) \to \Map_D^*(X,Y)$ induces  isomorphisms of homology groups
through dimension $n_D$ 
for complex projective spaces $X$ and $Y$,
%when $X$ and $Y$ are complex projective spaces, 
where
$\Hol_D^*(X,Y)$ (resp. $\Map_D^*(X,Y)$) denotes the space of 
base-point preserving holomorphic (reps. continuous) maps 
from $X$ to $Y$ of degree $D$.
%from $X$ to $Y$ with  degree $D$ and $\Map_D(X,Y)$
%(reap. $\Map_D^*(X,Y)$) denotes the corresponding space of 
%(resp. base-point preserving) continuous maps.
%where the meaning of \lq\lq degree\rq\rq will be explained later. 
\par
%%%%%%%(Page 2 The second line)%%%%
In \cite{AKY1} and \cite{KY4}  the case where the structure ${\cal S}$ is that of a real algebraic variety was considered,
and integers  $n_D$ were found, such that
the natural projection map $i_D:A_D(X,Y)\to \Map_D^*(X,Y)$ induces
isomorphisms of homology groups through dimension $n_D$
for real projective spaces $X$ and $Y$, where
$A_D (X,Y)$ is a  space of  tuples of polynomials representing  the elements  of 
$\Alg_D^* (X,Y)$ - the space of base-point preserving algebraic   (regular) maps from $X$ to $Y$ of degree $D$
 (we will refer to $A_D (X,Y)$ as an \lq\lq algebraic approximation\rq\rq\  to the mapping space $\Map_D^*(X,Y)$).
%%%%%
\par
%%%%
Recently Mostovoy and Munguia-Millanueva  generalized Mostovoy's earlier result 
\cite{Mo3} to the case of holomorphic maps  
from a complex projective space $\CP^m$ to
a compact smooth toric variety $\XS$ in \cite{MV},
where $\XS$ denotes the toric variety associated to a fan $\Sigma$.
%%%%%%%%%(This is OK until this part)%%
%%%%%%
\par\vspace{1mm}\par
%%%
%%%
In this paper, we study a real analogue of this result \cite{MV}
(a different kind of analogue will be studied in the subsequent paper \cite{KOY2}).
%%%  
More precisely, our original aim was to generalize the  results of \cite{KY5} to the spaces of algebraic maps from a real projective space $\RP^m$ 
to a compact smooth toric variety $\XS$.%
%%%%%%%%%%%%%%%%%%%%%
%%%%%(FootNote 2)%%%%%%
%%%%%%%%%%%%%%%%%%%%%
\footnote{%
%%%
Note that an algebraic map from a real variety $V$ to a complex variety $W$
is a morphism defined on an open dense subset of the complexification $V_{\C}$
of $V$ which contains all the $\R$-valued points of $V$.
%So an algebraic map from $\RP^m$ to $\XS$ is defined only on some open dense subset
%of $\CP^m$, and it cannot be extended to $\CP^m$ in general.
}
%%%%
%%%%(End of FootNote 2)%%%%%%%%%%
%\par
%%
Although our  approach is based on  the original ideas of Mostovoy (\cite{Mo2}, \cite{Mo3}), the real case presents special difficulties. 
%Some are of algebraic-geometric nature.
%%
%\par
%%%
For example,
what we call here an algebraic map from $\RP^m$ to $\XS$ is defined on some open dense subset
of $\CP^m$, and it cannot be extended to the entire complexification 
$\CP^m$ in general.
%%%
%\par
%%
Moreover, there is another difference between the complex case and the real one, which is the source of a greater difficulty. 
%%
%\par
%%
Namely, in the complex case, 
base-point preserving algebraic (equivalently holomorphic) maps from 
%a complex projective space  
$\CP^m$ to %a complex toric variety 
$\XS$ are determined by $r$-tuples of of homogenous polynomials 
taking values outside a certain subvariety of $\C^r$, 
where $r$ is the number of the Cox's homogenous coordinates of $\XS$.
%%%%
%\par
%%%
\par
On the other hand, in the real case, such
$r$-tuples of  polynomials determine algebraic maps only up to multiplication by certain positive valued homogenous polynomial functions.
%%% 
%%
This means that base-point preserving algebraic maps from $\RP^m$ to $\XS$
cannot be uniquely represented (in homogeneous coordinates) by $r$-tuples of polynomials with 
with a fixed \lq\lq degree\rq\rq\
since we can multiply any $r$-tuple by a positive valued homogeneous polynomial to obtain another tuple representing the same algebraic map.
%%
%\par
While an algebraic map from $\RP^m$ to $\XS$ has a uniquely defined  \lq\lq minimal degree\rq\rq\ 
%(cf. Definition \ref{dfn: degree})  
and the space of algebraic maps with a fixed minimal degree can be described in terms of $r$-tuples of homogenous polynomials of the same degree,  
%(see Prop. \ref{prp: the minimal degree}), 
the topology of the space of algebraic maps with a fixed minimal degree  is very complicated to analyze and
we shall not make here  any use of this concept 
(besides defining it).%
%%%%%%%%%%%%%%%%%%%%%%%%%
%%%(FootNote 3)%%%%%%%%%%
\footnote{%
%This is explained in \S \ref{section: Appendix}. 
See Proposition \ref{prp: the minimal degree} and Definition \ref{dfn: degree}.}
%%%%%%%%%%%%%%%%%%%%%%%%
%%%%
%%%%
%%%
\par\vspace{1mm}\par
%%%%
%Instead, we will follow the approach analogous to that in our earlier work \cite{AKY1} and will approximate the space of base-point preserving continuous maps 
%not by finite dimensional subspaces %of $\Map(\RP^m,X)$ 
%consisting of algebraic maps but by  finite dimensional spaces $A_D(m,\XS)$ 
%of $r$-tuples of homogeneous polynomials of a fixed degree $D$ representing 
%base-point preserving algebraic maps.
% from $\RP^m$ to $\XS$ with degree $D$.
%%
%The algebraic maps in this case are morphisms defined on a dense open subset of 
%$\mathbb{C}\P^m$ which contains all the $\mathbb{R}$-valued 
%points of $\mathbb{C}\P^m$.  Such maps can be represented as certain tuples of homogeneous polynomials %without common non-zero roots
%taking values (on $\RP^m$) outside a certain subvariety 
%(see Proposition \ref{prp: the minimal degree} below). 
%%%%%
%\par
%%%%% 
%For example, since 
%in general 
%an algebraic map from a real algebraic variety $X$ to a complex toric variety  
%$\XS$
%cannot be extended to the entire complexification $X_\Bbb C$ of $X$,
% we cannot use the Cox's  homogeneous coordinates directly as in \cite{MV} and 
% we need Proposition \ref{prp: degree} 
% %(Proposition \ref{prp: the minimal degree}) 
% in order to do so. 
%%%%%
%the one dimensional cones in a fan $\Sigma$.
%is a certain invariant of the toric variety $\XS$.%\footnote{%
%%%%%%%%%%%%%%%%%%
%%%(FootNote 1)%%%
%In fact,
%$r$ is the number of one dimensional cone in $\Sigma$.
%See section \ref{section 1.1} in detail.} 
%%%%%%
%\par
%%%%%%  
%%%
%%%
%%%%
\par
Now, let $i_D:A_D(m,\XS)\to \Map^*(\RP^m,\XS)$
denote the natural map 
given by sending
$r$-tuple of polynomials to its representing 
algebraic map,
and
consider the natural surjective projection
$\Psi_D:A_D(m,\XS) \to \Alg_D^*(\RP^m,\XS)=i_D(A_D(m,\XS))$
on its image.
%We call an element $f\in \Alg_D^*(\RP^m,\XS)$
%as a base-point preserving algebraic maps of a given degree $D$ (this \lq\lq degree\rq\rq\ has to be distinguished form the \lq\lq minimal degree\rq\rq).
%%
Because any fibre of $\Psi_D$ is contractible,
%the map $i_D:A_D(m,\XS)\to \Map_D^*(\RP^m,\XS)$ 
% (when viewed as a map onto its image $\Alg_D^*(\RP^m,\XS)$) 
%has a contractible fibre. 
if we could prove that this map is a quasi-fibration, it would imply that it is a homotopy equivalence and we could then just imitate the method of \cite{MV}. 
% (for our purposes 
% it would even be enough to prove that the map is acyclic on homology groups). 
 Unfortunately, it seems  difficult to prove that this map is a quasi-fibration or satisfies some other condition that leads to the conclusion that it is a
 homotopy equivalence.
 %the space of tuples of polynomials is homotopy or homology equivalent to the space of algebraic maps represented by them.
%%
\par
%%%%%
To get around this problem we adopt the approach used in \cite{AKY1}. 
Namely, 
we restrict ourselves to considering only the spaces of $r$-tuples of polynomials which represent algebraic maps rather than the spaces of algebraic maps themselves.  These will be our finite dimensional approximations to the space of all continuous maps. 
(We conjecture that both kinds of algebraic approximations, by tuples of polynomials and by algebraic maps which they determine,  are homotopy equivalent). 
%%%
Because we cannot work with spaces of algebraic maps we cannot make use of the real analogue of 
 \cite[Prop. 3]{MV} (although such an analogue can be proved by a similar method).
%%%
\par
Instead, we prove a similar theorem (Theorem \ref{thm: AKY1-stable}), where in place of the stabilized space of algebraic maps we use the stabilized space of tuples of polynomials representing them and obtain a homology equivalence instead of a homotopy one with the space of all continuous maps. 
Our method of proof is more complicated since we have to rely on the spectral sequence constructed by Vassiliev \cite{Va} for computing the homology of the space of continuous maps from a $q$-dimensional CW complex to a $q$-connected one 
(see \S \ref{section: Vassiliev} for more details).%
%%%%%%%%%%%%%%%%%%%
% %%%%(FootNote )%%%
%\footnote{ 
%This is explained in \S \ref{section: Vassiliev} for more details.}
%%%%%%%%%%%%%%%%%%%%%%
%The requirement that this spectral exists imposes the condition
%%%%%%%%%%% 
%%%%(0.1)%%
%\begin{equation}\label{condition}
%%%%%%%%%%%
%2\leq m \leq 2(\rmin -1)
%\end{equation}
%%in our main result (cf. (\ref{dim rmin}), Theorem \ref{thm: I},  Lemma \ref{lemma: Vass}).
%%%%%%
\par\vspace{1mm}\par
%%%
More precisely, our proof makes use of three spectral sequences corresponding  to three kinds of simplicial resolutions of discriminants;  the first one defined by Vassiliev, the other two, its variants due to Mostovoy.
\par  
The first resolution, called by Vassiliev 
\lq\lq the simplicial resolution\rq\rq\  (and  in this paper the \lq\lq non-degenerate simplicial resolution\rq\rq\  to distinguish it from the other two),  has  a $(k-1)$-simplex as its fibre
whenever the inverse image %of a point 
consists
of $k$ points. Vassiliev used this resolution to construct a spectral sequence for computing  cohomology of  spaces of continuous mappings.
%%% 
The same kind of resolution and a corresponding spectral sequence can be defined for the space of algebraic maps and clearly there is a natural map  between the resolutions which induces a homomorphism between the spectral sequences. However, this is not enough to prove our result as the spectral sequence has too many non-vanishing terms, due to the fact that the values of an algebraic map of fixed degree at different points are not necessarily independent. 
Mostovoy's idea in \cite{Mo2} was to use a degenerate simplicial resolution (which we have decided to name 
{\it the Veronese resolution} because it it based on the use of a Veronese-like embedding). This resolution collapses some of the fibres of the non-degenerate resolution and makes the corresponding terms 0.  The problem with the Veronese resolution is that it does not map into the Vassiliev's non-degenerate simplicial resolution of the space of continuous maps. However, by combining the two resolutions we can prove  Theorem \ref{thm: AKY1-stable}. 
\par
%%%
It may be worth noting that  by using only these two resolutions we can prove Theorem \ref{thm: AKY1}, which differs from our main Theorem \ref{thm: I} only in giving a worse stabilization dimension. To obtain a stronger result we turn to another idea of Mostovoy: that of the (non-degenerate) simplicial resolution truncated after some term. 
%%%
\par
%\vspace{1mm}\par
%%%%
%More precisely,
The basic idea behind the proofs is as follows. We want to compare the topology of two successive algebraic approximations. We make use of the fact that these spaces are complements of discriminants in affine spaces and their topology can be related to that of discriminants by Alexander duality.  
To compare the homology we resolve the singularities of 
both spaces using the non-degenerate simplicial resolution.  
The non-degenerate resolution has a natural filtration, 
the $k$-th term of which is the union of the $(k-1)$-th skeleta of the simplices in the fibres of the non-degenerate resolution. 
The first few terms of the filtration are easy to describe but after a certain dimension 
they become intractable. 
To deal with this problem we truncate the resolution after this dimension, 
by taking only skeleta of lower dimensions 
and collapsing non-contractible fibres (thus obtaining again a space homotopy equivalent to the discriminant). This truncated resolution inherits a filtration, which turns out to be much more manageable. 
  A comparison of these truncated resolutions and their spectral sequences gives  our main result.  
 \par\vspace{1mm}\par
%%
%\par
This paper is organized as follows.
In \S \ref{section 1}, we define the basic notions of toric varieties and state our main results. 
In \S \ref{section 2}, we recall  the definitions of  the various simplicial resolutions used in this paper.
In \S \ref{section 3} we study the spectral sequence induced %by 
from the
non-degenerate simplicial resolution of  discriminants,
and in \S \ref{section 4} we prove our key stability result (Theorem \ref{thm: sd}).
In \S \ref{section: polyhedral product} we recall the elementary results of polyhedral
products and investigate the connectivity of the complement $\C^r\setminus \ZS$.
In \S \ref{section: main result}, we give the proofs of Theorem \ref{thm: AKY1-stable}
and the main result (Theorem \ref{thm: I}) by using Theorem \ref{thm: AKY1}. 
In \S \ref{section: Vassiliev} we prove Theorem \ref{thm: AKY1} by using
the Vassiliev spectral sequence \cite{Va} and that induced from the Veronese (degenerate) resolution. 
%In \S \ref{section: polyhedral product}
%we recall polyhedral products and give the proof of Lemma \ref{lmm: connectivity}.
In \S \ref{section: Appendix} we describe  some
non-trivial facts, define the minimal degree of algebraic maps, and consider the problem concerning the relationship between  spaces of polynomial tuples and the spaces of algebraic maps induced by them. 
 %%%
%The related problem concerning this problem is quoted in \S \ref{section: Appendix}.
%\par
% We shall study the corresponding spaces of equivariant algebraic maps
%in the subsequent paper \cite{KOY2}.
%%%%
%%%%

%%%(SECTION 1)%%%%%%%%%%%%%%%%%%%%%%
\section{Toric varieties and the main results.}\label{section 1}
%%%%%%%%%%%%%%%%%%%%%%%%%%%%%%%%%%%%

%%%%(1.1. Basic notations and toric varieties)%%%%%
\subsection{Toric varieties and polynomials representing
algebraic maps.}\label{section 1.1}
%%%%%%%%%%%%%%%%%%%%%%%%%%%%%%%%%%%%%%%%%%%%%%%%%%
%%%%%%%%%%%%%%%%%%%%%
%%%
%%%
%%%
%%%%%%%%%%%%%%%%%%%%%%%%%%%%%%%%%
%%%%%%%%%(Toric varieties)%%%%%%%
\paragraph{Toric varieties. }
%%%%%%%%%%%%%%%%%%%%%%%%%%%%%%%%%
%We summarize the basic definitions and facts about toric varieties \cite{CLS}.
%%%%%
%\par\vspace{2mm}\par
%%%(i)%%%%
%(i)
%%%%
An irreducible normal algebraic 
variety $X$ (over $\C$) is called {\it a toric variety} if
it has an algebraic action of algebraic torus $\T^r=(\C^*)^r$,
such that the orbit $\T^r\cdot *$ of some point $*\in X$
is dense in $X$ and isomorphic to $\T^r$.
%%%%%
{\it A strong convex rational polyhedral cone} $\sigma$ in $\R^n$ is a subset
of $\R^n$ of the form
$\sigma =\mbox{Cone}(\{{\bf n}_k\}_{k=1}^s)=
\{\sum_{k=1}^s a_k{\bf n}_k\vert a_k\geq 0\},$
where
$\{{\bf n}_k\}_{k=1}^s\subset \Z^n$, which
does not contain any line.
%%%
A finite collection $\Sigma$ of strongly convex rational polyhedral cones
in $\R^n$ is called {\it a fan}
if every face of an element of $\Sigma$ belongs to $\Sigma$ and
the intersection of any two elements of $\Sigma$ is a face of each.
% other.
%%%
\par
%%%%
It is known that
a toric variety $X$ 
is completely characterized up to isomorphism by its fan
$\Sigma$. 
We denote by $\XS$
the toric variety associated to $\Sigma$.
\par
%%%%(ii)%%%
%(ii)
%%%
The dimension of a cone $\sigma$ is the minimal dimension of subspaces $W\subset \R^n$
such that $\sigma\subset W$.
A cone $\sigma$ in $\R^n$ is called {\it smooth} (reps. {\it simplicial})
if it is generated by a subset of a basis of
$\Z^n$ (resp. a subset of a basis of $\R^n$).
%% 
%%%%
\par
%%%%
A fan $\Sigma$ is called {\it smooth} (resp. {\it simplicial}) if every cone in $\Sigma$ is smooth
(resp. simplicial).
%%%
A fan is called {\it complete} if $\bigcup_{\sigma\in\Sigma}\sigma=\R^n$.
Note that $X_{\Sigma}$ is compact if and only if $\Sigma$ is complete
\cite[Theorem 3.4.1]{CLS}, and that
$\XS$ is a smooth toric variety if and only if $\Sigma$ is smooth
\cite[Theorem 1.3.12]{CLS}.
It is also known that $\pi_1(\XS)$ is isomorphic to the quotient of $\Z^n$
by the subgroup generated by $\bigcup_{\sigma \in \Sigma}\sigma \cap \Z^n$
\cite[Theorem 12.1.10]{CLS}.
In particular, if $\XS$ is compact, $\XS$ is simply connected.
%%%
\par
%%%(iii)%%%
%(iii)
%%%%%%%%%
% {\it A lattice polytope} $P$ in $\R^n$ is a 
%convex hull of a finite subset $S$ in $\Z^n$,
%and it is called {\it simple} (resp. {\it simplicial}) if every vertex of $P$ is the intersection
%of precisely
%$n$ facets of $P$
%%\cite[Def. 2.2.3]{CLS}
%(resp.
%if  every proper face of $P$ is a simplex).
%%\cite[Prop. 2.16]{Z}.
%For an $n$ dimensional lattice polytope $P$, we denote by
%$\Sigma_P$ {\it the normal fan}  of $P$ in $\R^n$
%\cite[page 76-77]{CLS}.
%%%%%
%%\par
%%%%%
%It is known that
% the toric variety
%$\XS$ is projective if and only if
%$\Sigma =\Sigma_P$ for some $n$ dimensional lattice polytope $P$ in $\R^n$
%\cite[Theorem 6.2.1 (cf. Theorem 4.2.1)]{CLS}.
%%Remark that a projective toric variety $\XS$ with $\Sigma =\Sigma_P$ is smooth
%%if and only if  the lattice polytope $P$ is 
%%simple and the primitive normal vectors ${\bf m}_1,\cdots ,{\bf m}_n$ of any set of $n$ 
%%facets $F_1,\cdots ,F_k$ meeting at the same vertex form a basis of $\Z^n$
%%\cite[5.1.3]{BP}.

%%%%%%%%%%%%%
%%%%%%%(Homogenous coordinates of toric varieties)%%%%
%%%%%%%%%%%%%%%%%%%%%%%
\paragraph{Homogenous coordinates on toric varieties.}
%%%%%%%%%%%%%%%%%%%%%%%%%%%%%%%%%

We shall use the symbols $\{z_k\}_{k=1}^r$ to denote variables of polynomials.
For $f_1,\cdots ,f_s\in \C [z_1,\cdots ,z_r],$
let $V(f_1,\cdots ,f_s)$ denote the affine variety
$
V(f_1,\cdots ,f_s)=\{\textbf{\textit{x}}  \in\C^r\ \vert \ 
f_k(\textbf{\textit{x}})=0\mbox{ for each }1\leq k\leq s\}
$
given by the polynomial equations
$f_1=\cdots =f_s=0$.
%%%%
\par\vspace{2mm}\par
%%%%(i)%%
%(i)
%%%%%%%%s
Let $\Sigma (1)=\{\rho_1,\cdots ,\rho_r\}$ denote the set of all
one dimensional cones ({\it rays}) in a fan $\Sigma$,
and for each $1\leq k\leq r$
let ${\bf n}_k\in \Z^n$ denote the generator of $\rho_k\cap \Z^n$
(called {\it the primitive element} of $\rho_k$)
such that $\rho_k\cap \Z^n=\Z_{\geq 0}\cdot {\bf n}_k$. 
%%%%%%
Let $\ZS\subset \C^r$ denote the affine variety defined by
%%%(1.1)%%%%
\begin{equation}\label{eq: ZS}
%%%%%%%%%%%%
\ZS =V(z^{\hat{\sigma}}\vert \sigma\in \Sigma),
\end{equation}
%%%%%%%%%%%%
where
$z^{\hat{\sigma}}$ denotes the monomial
%%%(1.2)%%%%
\begin{equation}\label{zs}
z^{\hat{\sigma}}=
\prod_{1\leq k\leq r,\n_k\not\in\sigma}z_k
\in \Z [z_1,\cdots ,z_r]
\quad\quad
(\sigma\in \Sigma).
\end{equation}
%%%%%%%%%%%
%%%%%%%%
\par
%%%(ii)%%
%(ii)
%%%%%%%%
Let $\GS\subset \T^r$ denote the subgroup defined by
%%%(1.3)%%%%%
\begin{equation}\label{GS}
%%%
\GS =\{(\mu_1,\cdots ,\mu_r)\in \T^r\ \vert \ 
\prod_{k=1}^r\mu_k^{\langle {\bf m},\n_k\rangle}=1
\mbox{ for all }{\bf m}\in\Z^n\},
%%%
\end{equation}
%%%%%%%%%%%
where we set
$\langle \textbf{\textit{x}},\textbf{\textit{y}}\rangle =\sum_{k=1}^n x_ky_k$
for
$\textbf{\textit{x}}=(x_1,\cdots ,x_n),\ \textbf{\textit{y}} =(y_1,\cdots ,y_n)\in \R^n$.
\par
%%%(iii)%%
%(iii)
%%%%%%%%%%
%%%%%%%%
\par\vspace{2mm}\par
%%%%%%%
We say that a set of primitive elements
$\{{\bf n}_{i_1},\cdots ,{\bf n}_{i_s}\}$ is  
{\it primitive} 
%(or {\it a primitive collection})
if they do not lie in any cone in $\Sigma$ but every
proper subset has this property. 
It is known (\cite[Prop. 5.1.6]{CLS}) that
%%%%%%%%%%
%%%(1.4)%%%%
\begin{equation}\label{pri}
%%%%%%%%%%%%
\ZS =\bigcup_{\{{\bf n}_{i_1},\cdots ,{\bf n}_{i_s}\}
: \ {\rm primitive}}
V(z_{i_1},\cdots ,z_{i_s}).
\end{equation}
%%%%%%%%%%%%
Note that $\ZS$ is a closed variety of real dimension $2(r-\rmin )$,
where we set
%%%(1.5)%%%
\begin{equation}\label{dim rmin}
%%%
r_{\rm min}=\min \big\{s\in \Z_{\geq 1}\ \vert \ 
\{{\bf n}_{i_1},\cdots ,{\bf n}_{i_s}\}
\mbox{ is primitive}\big\}.
%%%
\end{equation}
%%%
%%%%%
It is known  that if the set $\{\n_1,\cdots ,\n_r\}$ spans $\R^n$
and $\XS$ is smooth, there is an isomorphism 
(\cite[Theorem 5.1.11]{CLS}, \cite[Prop. 6.7]{Pa1})
%%%(1.6)%%%%%%
\begin{equation}\label{hs}
%%%%%%%%%%%%%
\XS \cong (\C^r\setminus \ZS)/\GS ,%
\end{equation}
%%%%%%%%%%%%%%
%%
%%%%%%%%%%%%
%%%%(1.6)%%%%%
%\begin{equation}\label{equ: hs}
%%%%
%\XS \cong (\C^r \setminus \ZS)//\GS ,
%%%%
%\end{equation}
%%%%%%%%%%%
where 
the group $\GS$ acts freely on the complement $\C^r\setminus \ZS$ by coordinate-wise multiplication. 
%and %(\ref{hs}) 
%the space $(\C^r \setminus \ZS)//\GS$ is called
%{\it the almost geometric quotient}. 
%Moreover, if $\XS$ is smooth, it is known that $\GS$ acts freely 
%on the complement $\C^r\setminus \ZS$,
% and there is an isomorphism (\cite[Prop. 6.7]{Pa1})
%%%%(1.7)%%%%%%
%\begin{equation}\label{hs}
%%%%%%%%%%%%%%
%\XS \cong (\C^r\setminus \ZS)/\GS.%
%\end{equation}
%%%%%%%%%%%%%%

%%%%%%%%%%%%%%%%%%%%%%%%%%%%%%%
%%(Example 1.1)%%
\begin{exa}\label{exa: Hirzebruch}
%%%%%%%%%%%%%%%%%
{\rm
For each $k\in \Z$, let $H(k)$ denote {\it the Hirzebruch surface} given by
%%%()%%%
$
H(k)=\{([x_0:x_1:x_2],[y_1:y_2])\in\CP^2\times\CP^1\ \vert \ x_1y_1^k=x_2y_2^k\}
\subset \CP^2\times\CP^1.
$
%%%%%%%%%%
Note that there are isomorphisms 
$H(0)\cong \CP^1\times\CP^1$ and $H(k)\cong H(-k)$ for $k\geq 1$. 
Let  $\Sigma$ denote the fan in $\R^2$ given by
$$
\big\{\mbox{Cone}({\bf n}_1,{\bf n}_2),\mbox{Cone}({\bf n}_2,{\bf n}_3),
\mbox{Cone}({\bf n}_3,{\bf n}_4),\mbox{Cone}({\bf n}_4,{\bf n}_1),
\R_{\geq 0}\cdot {\bf n}_j\ (1\leq j\leq 4),\ {\bf 0}\big\},
$$ 
where
${\bf n}_1=(1,0)$, ${\bf n}_2=(0,1)$, ${\bf n}_3=(-1,k)$ and
${\bf n}_4=(0,-1)$.
Then we can  easily see that $H(k)=X_{\Sigma}$ \cite[Example 3.1.16]{CLS}.
Since $\{{\bf n}_1,{\bf n}_3\}$ and $\{{\bf n}_2,{\bf n}_4\}$ are primitive, $\rmin =2$.
Moreover,
by using (\ref{GS}), (\ref{pri}) and (\ref{hs})
we also obtain the isomorphism
%$H(k)\cong (\C^4\setminus \ZS)/\GS$, where}
%%%(1.7)%%
\begin{equation}\label{H(k)}
%\begin{cases}
H(k)\cong
%(\C^4\setminus\ZS)/\GS =
\{(y_1,y_2,y_3,y_4)\in\C^4\ \vert \ (y_1,y_3)\not= (0,0),(y_2,y_4)\not= (0,0)\}/\GS,
%\\
%\GS &=\{(\mu_1,\mu_2,\mu_1,\mu_1^k\mu_2)\ \vert \ \mu_1,\mu_2\in \C^*\}\cong \T^2.
%\end{cases}
\end{equation}
where $\GS =\{(\mu_1,\mu_2,\mu_1,\mu_1^k\mu_2)\ \vert \ \mu_1,\mu_2\in \C^*\}\cong \T^2.$
\qed
}
\end{exa}
%%%%
%%%%%(Example 1.2)%%%%%%
%\begin{exa}\label{exa: projective space}
%%%%%%%%%%%%%%%%%%%%%%%%
%{\rm
%Let $\textbf{\textit{e}}_j$ $(1\leq j\leq n)$ denote the standard basis of $\Z^n$
%and define the elements $\{{\bf n}_j\ \vert \ 0\leq j\leq n\}$ by
%${\bf n}_{0}=-\sum_{k=1}^n\textbf{\textit{e}}_k$ and
%${\bf n}_j=\textbf{\textit{e}}_j$ for $1\leq j\leq n$.
%Let $\Sigma$ denote the fan in $\R^n$ given by  all cones generated by all proper subsets of
%$\{{\bf n}_k\  \vert \ 0\leq k\leq n\}$.
%Then we can easily identify $\XS =\CP^n$ \cite[Example 3.1.10]{CLS}.
%Since $\{{\bf n}_j\  \vert 0\leq j\leq n\}$ is primitive,  $\rmin =n+1$.
%Moreover, an easy computation shows that
%%%()%%%
%$\ZS =\{{\bf 0}\}$ and
%$\GS =\{(\mu,\mu,\cdots ,\mu)\ \vert \ \mu\in\C^*\}\cong \C^*.$
%Hence, there is an isomorphism $\CP^n\cong (\C^{n+1}\setminus \ZS)/\GS
%=(\C^{n+1}\setminus \{{\bf 0}\})/\C^*$, which coincides with a usual
%homogenous coordinate description  of $\CP^n$.
%}
%%%%%
%\end{exa}
%%%%%%

%%%%
\paragraph{Spaces of mappings. }
%%%%%%%%%%%%%%%%%%%%%%%%%%%%%%%%
For connected spaces $X$ and $Y$,
let $\Map (X,Y)$ be the space of all
continuous maps $f:X\to Y$ and  
 $\Map^* (X,Y)$  the corresponding subspace
 of all based continuous maps.
%%%
If $m\geq 2$ and $g\in \Map^* (\RP^{m-1},X)$, let
$F(\RP^m,X;g)$ denote the subspace of $\Map^*(\RP^m,X)$ given by
%%)%%
\begin{equation*}%\label{H(k)}
%%%%%%%%%
F(\RP^m,X;g)=\{f\in \Map^*(\RP^m,X)\ \vert \ \  f\vert \RP^{m-1}=g\},
\end{equation*}
%%%%%%%%
where we identify $\RP^{m-1}\subset \RP^m$ by putting
$x_m=0$.
It is known that there is a homotopy equivalence
$F(\RP^m,X;g)\simeq \Omega^mX$ if it is not an empty set
(see Lemma \ref{lmm: A1}).
%\footnote{%
%%%%%%%%%%%%%%%%%%%%%%%%%%%%%%%%%%%%%%%%%%
%%%(FootNote )%%%%%%%%%%%%%%%%%%%%%%%%%%%%%
%See Lemma \ref{lmm: A1}.}
%%%%%%%%%%%%%%%%%%%%%%%%%%%
%%%%

%%%%%%%%%%%%%%%%%%%%%%%%%%%%%%%%%%%%%%%%%%%%%
%%%%(1.2 Spaces of polynomials and the main results)%%
%\subsection{Spaces of polynomials and spaces of algebraic maps.}
%%%%%%%%%
%%%
%%%%(Basic notations and assumptions)%%%
\paragraph{Assumptions. }
From now on, we adopt the following notational conventions and two assumptions:
%%%%%%%%%%%%%%%%%%%%%%%
%%%(Assumptions)%%%%%%%%
\begin{enumerate}
%%%%%%%%%%%%%%%%%(1.7.1), (1.7.2)%%
\item[(\ref{H(k)}.1)]
Let
$\Sigma$ be a complete smooth fan in $\R^n$,  
$\Sigma (1)=\{\rho_1,\cdots ,\rho_r\}$ the set of all one-dimensional cones in $\Sigma$, and suppose that the primitive elements
$\{{\bf n}_1,\cdots ,{\bf n}_r\}$  span
$\R^n$, where
${\bf n}_k\in \Z^n$ denotes the primitive element of
$\rho_k$ for $1\leq k\leq r$.
%%%
\item[(\ref{H(k)}.2)]
Let
$D=(d_1,\cdots ,d_r)$ be an $r$-tuple of integers
such that $\sum_{k=1}^rd_k{\bf n}_k={\bf 0}$.
%%
%%%
%%
%%%%%%%%
\end{enumerate}
%%%%
\par\vspace{2mm}\par
%%%%
Then, by (\ref{H(k)}.1) we can make the identification
$\XS =(\C^r\setminus \ZS)/\GS$ as in (\ref{hs}).
%%
%%%%
%%
For each $(a_1,\cdots ,a_r)\in \C^r\setminus \ZS$, we denote by
$[a_1,\cdots,a_r]$ the corresponding element of $\XS$.
%%
%%%%

%%%%%%(i)%%%
%(i)
\paragraph{Spaces of polynomials.}
%%%%%%%%%
Let ${\cal H}_{d,m}\subset \C [z_0,\ldots ,z_m]$
denote the space of global sections
$H^{0}(\CP^m,{\cal O}_{\CP^m}(d))$ of the %canonical 
line bundle
${\cal O}_{\CP^m}(d)$ of degree $d$.
Note that
the space
${\cal H}_{d,m}\subset \C[z_0,\cdots ,z_m]$ coincides with
the subspace consisting of all
homogeneous polynomials of degree $d$ if $d\geq 0$ and that
${\cal H}_{d,m}=0$ if $d<0$.%
%%%%%%%%
%%%(Footnote 4)%%%%%%
\footnote{This is because $H^{0}(\CP^m,{\cal O}_{\CP^m}(d))=0$ if $d<0$.}
%%%%%%%%%%%%%
%%%
%%%
For each $r$-tuple $D=(d_1,\cdots ,d_r)\in \Z^r$,
let $A_D(m)$ denote the space
%%%(1.8)%%%
\begin{equation}\label{AD}
%%%%%%%%%%%
A_D(m)={\cal H}_{d_1,m}\times {\cal H}_{d_2,m}\times \cdots \times
{\cal H}_{d_r,m}
%%%%%%%%%%%
\end{equation}
%%%%%%%%%%%%
%%%%
and
let $A_{D,\Sigma}(m)\subset A_D(m)$ denote the subspace defined by
%%%(1.9)%%%
\begin{equation}\label{ADSigma}
%%%%%%%%%%
A_{D,\Sigma}(m)=\big\{(f_1,\cdots ,f_r)\in A_D(m)\ \vert \
F(\textbf{\textit{x}})
\notin \ZS
\mbox{ for any }\textbf{\textit{x}}\in \R^{m+1}\setminus
\{{\bf 0}\}\big\},
\end{equation}
%%%%%%%%%
where we  write $F(\textbf{\textit{x}})=(f_1(\textbf{\textit{x}}),\cdots,
f_r(\textbf{\textit{x}}))$
for $\textbf{\textit{x}}\in \R^{m+1}\setminus \{{\bf 0}\}$.
%%%%%%%
\par\vspace{1mm}\par

%%%%%%%
Next, we %want to 
define a map
$j_D^{\p}:A_{D,\Sigma}(m) \to \Map (\RP^m,\XS)$ by the formula
%%%(1.10)%%%%
\begin{equation}\label{def of j1D}
%%%%%%%%%%%%%%%%%%%%
j_D^{\p}(f_1,\cdots ,f_r)([\textbf{\textit{x}}])=
[f_1(\textbf{\textit{x}}),\cdots ,f_r(\textbf{\textit{x}})]
\quad
\mbox{for }\textbf{\textit{x}}\in \R^{m+1}\setminus\{{\bf 0}\}.%
\end{equation}
%%%%%%%
%We need to check the following: 
%%%%%%%%%%%%%%%
%%%%%%%%%%%%%%%
%%(Remark 1.2)%%%%%%%%%%
\begin{rmk}\label{lmm: A2}
%%%%%%%%%%%%%%%%
{\rm
Note that
the map 
$j_D^{\p}$
is well defined 
because $\sum_{k=1}^rd_k{\bf n}_k={\bf 0}$.
%% 
%%%%%%%%%%%%
%\begin{proof}
%%%%%%%%%%%%
In fact,
if $\lambda\in \R^*$,
then since
$\prod_{k=1}^r (\lambda^{d_k})^{\langle {\bf m},{\bf n}_k\rangle}=
\lambda^{\langle {\bf m},\sum_{k=1}^rd_k{\bf n}_k
\rangle}=1$
for any ${\bf m}\in \R^n$,
$(\lambda^{d_1},\cdots ,\lambda^{d_r})\in \GS$.
%%%%
So for $(f_1,\cdots ,f_r)\in A_{D,\Sigma}(m)$,
$(f_1(\lambda \textbf{\textit{x}}),\cdots ,f_r(\lambda \textbf{\textit{x}}))=
(\lambda^{d_1}f_1(\textbf{\textit{x}}),\cdots ,\lambda^{d_r}f_r(\textbf{\textit{x}})).$ Hence
$
[f_1(\lambda \textbf{\textit{x}}),\cdots ,f_r(\lambda \textbf{\textit{x}})]=
[f_1(\textbf{\textit{x}}),\cdots ,f_r(\textbf{\textit{x}})]$
in
$\XS$
for any $(\lambda ,\textbf{\textit{x}})\in \R^*\times (\R^{m+1}\setminus
\{{\bf 0}\})$ and the map $j_D^{\prime}$  is well-defined.
\qed
}
%%%%%%%%%%%
%\end{proof}
\end{rmk}

%%%%
Since the space $A_{D,\Sigma}(m)$ is connected, the image of $j^{\p}_D$ lies in a connected
component of $\Map (\RP^m,\XS)$, which will be denoted by
$\Map_D(\RP^m,\XS)$.
%%%%%%%%%%%%%%%%%%%%%%%%
This gives a natural map
%%%%%%%%%%%%%%%%%%%%%%%%
%%%(1.11)%%%%%%%%%%%%%%
\begin{equation}\label{id}
%%%%%%%%%%%%
j_D^{\p}:A_{D,\Sigma}(m) \to \Map_D(\RP^m,\XS).
%%%%%%
\end{equation}
%%%%%%%%

%\par\vspace{2mm}\par

%%%%%%%%%%%%%%%%%%%%%%%
\paragraph{Algebraic maps from a real algebraic variety to a complex algebraic variety.}
%%%%%%%%%%%(Remark 1.3)%%%%%%%
%\begin{rmk}\label{rmk: algebraic map}
%%%%%%%%%%%%%%%%%%%%%%%%%%%%%%
%{\rm
Let $X$ be an algebraic variety over 
%the field 
$\mathbb{R}$,
%of real numbers,
and $Y$ an algebraic variety  over 
%the field 
$\mathbb{C}$.
%of complex numbers.
We set $X_{\mathbb{C}}=X\times_{\mathbb{R}}\mathbb{C}$,
and denote by $X(\mathbb{R})$ the set of $\mathbb{R}$-valued points of $X$.
%\par
Let $\Phi:X_{\mathbb{C}} - \to Y$ be a rational map%
%%%%%%%%%%%%%%%%%%%%
%%%(Footnote 6)%%%%%
\footnote{%
The notation $\Phi:X_{\mathbb{C}}- \to Y$ is used to stress 
the fact that $\Phi$ is not a map, namely is not necessarily defined at every point of $X_{\mathbb{C}}$
%%%%%%%%%
%%%%%%%%
but is defined only on some Zariski dense open subset of $X_{\mathbb{C}}$.}
%%%%%
 defined over $\mathbb{C}$
and let $U$ be the largest Zariski dense open subset of $X_{\mathbb{C}}$
such that $(U,\varphi)$ represents $\Phi$, where $\varphi:U\to Y$ is a regular map
defined over $\mathbb{C}$.
If $X(\mathbb{R})\cap (X\setminus U)=\emptyset$, then we have a map
$\varphi|_{X(\mathbb{R})}:X(\mathbb{R})\to Y$.

An algebraic map $f:X\to Y$ is defined to be a map $X(\mathbb{R})\to Y$, which is also denoted by $f$
by abuse of notation, such that $f=\varphi|_{X(\mathbb{R})}$,
where $\varphi:U\to Y$ is a regular map defined over $\mathbb{C}$ 
on a Zariski open subset $U$ of $X_{\mathbb{C}}$ and $U$ contains $X(\mathbb{R})$. 

Note that an algebraic map $f:X\to Y$ is not a morphism of schemes,
although the algebraic varieties $X$ and $Y$ are (or can be regarded as) schemes and we use the adjective ``algebraic":
in fact, there does not exist a morphism from a variety defined over $\mathbb{R}$ with some $\mathbb{R}$-valued points to a variety defined over $\mathbb{C}$.
Note also that $\Phi:X_{\mathbb{C}} - \to Y$ is not only a meromorphic map but also a rational map. 
In particular, the map $\varphi$ is not only a holomorphic map but also a regular map. In this sense, 
an algebraic map $f:X\to Y$ defined above is indeed \lq\lq algebraic\rq\rq.
%}
%%%%%%
%\end{rmk}
%%%%%%()%%%
\par
%\vspace{2mm}\par
%%%
Now let us consider the case
$(X,Y)=(\RP^m,\XS)$.
%%%%%%%%%%%%
%%%
Let 
$C_D(m)$  denote the {\it incidence correspondence}
%%%(1.12)%%
\begin{equation}
C_D(m)=\{(F,\textbf{\textit{x}})\in
A_{D,\Sigma}(m)\times\C^{m+1}
\ \vert \ F(\textbf{\textit{x}})\in \ZS\}
\end{equation}
%%%%%
and $pr_1:C_D(m)\to A_{D,\Sigma}(m)$ the first projection.
Because the map $C_D(m)\to \Z$ given by
$(F,\textbf{\textit{x}})\mapsto\dim_{\C}pr_1^{-1}(F)$
is upper semicontinuous, the subspace
$A_{D,\Sigma}(m)^{\circ}$
is a Zariski open subspace of $A_{D,\Sigma}(m)$, where
%%%
%%%(1.13)%%
\begin{equation}\label{AD-circle}
%%%%%%
A_{D,\Sigma}(m)^{\circ}=
\{F\in A_{D,\Sigma}(m)\ \vert 
\ \dim_{\C}pr_1^{-1}(F)<m\}.
\end{equation}
%%%%%
%%
Note that we can see that every algebraic map $f:\RP^m\to \XS$ can be represented as
$f=j_D^{\p}(f_1,\cdots ,f_r)$ for some $D=(d_1,\cdots ,d_r)$ and
$(f_1,\cdots ,f_r)\in A_{D,\Sigma}(m)$ such that
$\sum_{k=1}^rd_k{\bf n}_k={\bf 0}$.
If $(f_1,\cdots ,f_r)\notin A_{D,\Sigma}(m)^{\circ},$
the representation of $f$  the degree $D$ are not unique.%
%%%%%
%%%%%%(FootNote 7)%%%
\footnote{%
%For example, if we set $\tilde{g}=\sum_{k=0}^mz_k^2$ and
%$\textbf{\textit{a}}=(a_1,\cdots ,a_r)\in (\Z_{\geq 1})^r$ with
%$\sum_{k=1}^ra_k{\bf n}_k={\bf 0}$, the equality
%$f=j_D^{\p}(f_1,\cdots ,f_r)=j^{\p}_{D+\textbf{\textit{a}}}
%(\tilde{g}^{a_1}f_1,\cdots ,\tilde{g}^{a_r}f_r)$ holds
See Remark \ref{rmk: homeo} (ii) (cf. Example \ref{exa: example}) for the details.}
%%%%%%%%%%%%%%%%%%%%
%\par
%%%%%%%%%%%%%%%%%%%
However, the following holds.%
%%%(FootNote 8)%%
\footnote{
The proof of Proposition \ref{prp: the minimal degree}
is given in \S \ref{section: Appendix}.}
%%%%%%%%%%%%%%%%%%%%%

%%%%(Proposition 1.3)%%%%%%%
\begin{prp}%[Proposition \ref{prp: algebraic map}]
\label{prp: the minimal degree}
%%%%%%%%%%%%%%%%%%%%%%%%%%%%
Let 
$X_{\Sigma}$ be a 
smooth compact complex toric variety
%satisfying (\ref{hs}.1)
and  let $f:\RP^m \to \XS$ be an algebraic map.
Then there exists a unique $r$-tuple $D=(d_1,\dots,d_r)\in \Z^r$
such that $\sum_{k=1}^rd_k{\bf n}_k={\bf 0}$ and that
$
f=j_D^{\p}(f_1,\cdots ,f_r)=[f_1,\cdots ,f_r]$
for some
$(f_1,\cdots ,f_r)\in 
A_{D,\Sigma}(m)^{\circ}$, where $f_k\in H^0(\CP^m,\mathcal{O}(d_k))={\cal H}_{d_k,m}$
for each $1\leq k\leq r$.
Moreover, the element $(f_1,\cdots ,f_r)\in A_{D,\Sigma}(m)^{\circ}$
is uniquely determined by the map $f$ up to $\GS$-action$;$ i.e.
if
$f=j_D^{\prime}(h_1,\cdots ,h_r)$  
for another
$(h_1,\cdots ,h_r)\in A_{D,\Sigma}(m)^{\circ}$, then
there exists  an element $(\mu_1,\cdots ,\mu_r)\in\GS$
such that $(h_1,\cdots ,h_r)=(\mu_1f_1,\cdots ,\mu_rf_r).$
%\qed
%%%%
%%%%%%%
\end{prp}
%%%
%%%%(Remark 1.4)%%
\begin{rmk}\label{rmk: AD}
%%%%
{\rm
The above result  is the analogue of
\cite[Theorem 3.1]{C} in the case of algebraic maps, and
by using it we can define a unique \lq\lq minimal degree\rq\rq of an algebraic map
(as in Definition \ref{dfn: degree}).
However, because the topology of the space of algebraic maps of fixed minimal degree is difficult to analyze, we do not use the concept of minimal degree beyond this point in this paper.
\qed
}
\end{rmk}
%%%%%%%%%%

%%%%%
Because ${\bf e}=(1,1,\cdots ,1)\in \C^r\setminus \ZS$, we can choose 
$x_0 =[1,1,\cdots,1]\in \XS$ as the base-point of $\XS$. 
%%%
Let $A_D(m,\XS)\subset A_{D,\Sigma}(m)$ be defined by
%%%%
%%%(1.14)%%
\begin{equation}\label{AD}
%%%%
A_D(m,\XS) =\{(f_1,\cdots ,f_r)\in A_{D,\Sigma}(m)\ \vert\ 
(f_1(\textbf{\textit{e}}_1),\cdots ,f_r(\textbf{\textit{e}}_1))={\bf e}\},
%%%
\end{equation}
%%%%%
where $\textbf{\textit{e}}_1=(1,0,\cdots ,0)\in \R^{m+1}$. We choose
 $[\textbf{\textit{e}}_1]=[1:0:\cdots :0]$ as the base-point of
$\RP^m$.
%%%%%%%%%
%%%%%%%%%
Note that 
$j_D^{\p}(f_1,\cdots ,f_r)\in \Map^*(\RP^m,\XS)$ 
if $(f_1,\cdots ,f_r)\in A_D(m,\XS)$.
%%%%%%%
%\par
%%%%%%
Hence, setting
 $\Map^*_D(\RP^m,\XS)=\Map^*(\RP^m,\XS)\cap
\Map_D(\RP^m,\XS)$, we have a map
%%%%%%%
%%%%%%%(1.15)%%%
\begin{equation}\label{def of iD}
i_D=j_D^{\p}\vert A_D(m,\XS):A_D(m,\XS) \to \Map^*_D(\RP^m,\XS).
\end{equation}
%%%%%%%%%%
\par\vspace{2mm}\par
%%%%%(ii)%%%%%%%%%%
%(ii)
%%%%%
Suppose that $m\geq 2$ and let 
us choose a fixed element $(g_1,\cdots ,g_r)\in A_{D}(m-1,\XS).$
For each $1\leq k\leq r$, let
$B_k=\{g_k+z_mh:h\in {\cal H}_{d_k-1,m}\}$.
Let 
$A_D(m,\XS;g)\subset A_D(m,\XS)$ be the subspace
%%%%%%%%
%%%%%%(1.16)%%%
\begin{equation}\label{ADR}
%%%%
A_D(m,\XS ;g)=A_D(m,\XS)\cap (B_1\times B_2\times \cdots \times B_r).
\end{equation}
%%%%%%%%%%%%%%
It is easy to see that
$i_D(f_1,\cdots ,f_r)\vert \RP^{m-1}=g$ if
$(f_1,\cdots ,f_r)\in A_D(m,\XS;g)$, where
$g$ denotes the element of $\Map^*_D(\RP^{m-1},\XS)$ given by
$g([\textbf{\textit{x}}])=
[g_1(\textbf{\textit{x}}),\cdots ,g_r(\textbf{\textit{x}})]
$
for
$\textbf{\textit{x}}\in \R^{m}\setminus \{{\bf 0}\}.$
%%%%%
%\par
%%%%%
Let 
$i_D^{\p}:A_D(m,\XS;g)\to F(\RP^m,\XS;g)\simeq \Omega^m\XS$ be the map defined 
by
%%%%%(1.17)%%%%
\begin{equation}\label{resAD}
%%%%%%%%%%
i_D^{\p}=i_D\vert A_D(m,\XS;g):A_D(m,\XS;g)\to F(\RP^m,\XS;g)\simeq \Omega^m\XS.
%%%%%%%%%
\end{equation}
%%%%%%%%%%%%%

%%%%%%%(The action of $\GS$ on the space $A_{D,\Sigma}(m)$)%%%%
\paragraph{The action of $\GS$ on the space $A_{D,\Sigma}(m)$
and its orbit space.}
%%%%%%%%%%%%%%%%%%%%%%%%%%%%%%%%%%%%%%%%%%%%%%%%%%%%%%%%%%%%%%
The group $\GS$ acts on the space
$A_{D,\Sigma}(m)$ by coordinate-wise multiplication.
%and  that the subspace $A_{D,\Sigma}(m)^{\circ}$ is
%$\GS$-invariant.
Let 
$\widetilde{A_D}(m,\XS)$ 
%and its subspace
%$\widetilde{A_D}(m,\XS)^{\circ}$ 
denote the orbit space
%%%(1.18)%%%%
\begin{equation}\label{ADS}
%%%%%%%%
\widetilde{A_D}(m,\XS)=A_{D,\Sigma}(m)/\GS .
%\quad
%\widetilde{A_{D}}(m,\XS)^{\circ}=A_{D,\Sigma}(m)^{\circ}/\GS.
\end{equation}
%%
%%%%%%%
Clearly, $j'_D$ also induces the map
%%%%()%%%
%\begin{equation}
%%%%%%%%%%%%
$j_D:\widetilde{A_D}(m,\XS) \to \Map_D(\RP^m,\XS)$
given by
%%%%
%%%%(1.19)%%%
\begin{equation}\label{def of jD}
%%%%%%%%
j_D([f_1,\cdots ,f_r])([\textbf{\textit{x}}])=
[f_1(\textbf{\textit{x}}),\cdots ,f_r(\textbf{\textit{x}})]
\quad
%\mbox{for }\textbf{\textit{x}}
%=(x_0,\cdots ,x_m)\in \R^{m+1}\setminus\{{\bf 0}\}.
%%%%%%%%%
\end{equation}
%%%%%%
for
$\textbf{\textit{x}}
=(x_0,\cdots ,x_m)\in \R^{m+1}\setminus\{{\bf 0}\}.$
%%%
%%(Degree of algebraic maps)%%%
\paragraph{Spaces of algebraic maps.}
%%%
Let
$\Alg (\RP^m,\XS)$ denote the space of all algebraic maps from
$\RP^m$ to $\XS$ and let
$\Gamma_D^{\p}:A_{D,\Sigma}(m)\to \Alg (\RP^m,\XS)$ 
denote the natural projection
given by
%%%%()%%
\begin{equation*}\label{equ: gamma}
\Gamma^{\p}_D(f_1,\cdots ,f_r)=j_D^{\p}(f_1,\cdots ,f_r)
=[f_1,\cdots ,f_r]
\qquad
\mbox{for }(f_1,\cdots ,f_r)\in A_{D,\Sigma}(m).
\end{equation*}
%%%%
Let $\Alg_D(\RP^m,\XS)$ denote its image
%%(1.20)%%%
\begin{equation} 
\Alg_D(\RP^m,\XS)=\Gamma^{\p}_D(A_{D,\Sigma}(m))
\subset \Map_D(\RP^m,\XS).
\end{equation}
%%%%%%%%
We can identify $\Gamma^{\p}_D$ with the  projection map
%%(1.21)%%
\begin{equation}
\Gamma_D^{\p}:A_{D,\Sigma}(m)\to \Alg_D(\RP^m,\XS).
\end{equation}
%%%%%%
Since the action of $\GS$ is compatible with $\Gamma^{\p}_D$,
it induces the natural projection map
%%%%%%%
%%%(1.22)%%
\begin{equation}\label{projection_gamma}
%%%%%%%%
\Gamma_D:\widetilde{A_D}(m,\XS) \to \Alg_D(\RP^m,\XS).
\end{equation}
%%%%%%%%%%%%
It is easy to see that
%%%(1.23)%%
\begin{equation}\label{equ: projection GammaD}
j_D=j_D^{\C}\circ \Gamma_D:
\widetilde{A_D}(m,\XS)\to \Map_D(\RP^m,\XS),
\end{equation}
%%%%%
where
$j_{D}^{\C}:\Alg_D(\RP^m,\XS) \stackrel{\subset}{\rightarrow}
 \Map_{D}(\RP^m,\XS)$ denotes
the inclusion map.
%%%
\par\vspace{2mm}\par

%%%
%%%
%%%
Let $\Alg^*_D(\RP^m,\XS)$ denote the subspace of 
$\Alg_D(\RP^m,\XS)$
given by $\Alg^*_D(\RP^m,\XS)= \Alg_D(\RP^m,\XS)\cap\Map^*(\RP^m,\XS)$.
\par\vspace{2mm}\par
%%%%%%%%%%
If $m\geq 2$ and $g\in \Alg_D^*(\RP^{m-1},\XS)$,
we denote by $\Alg^*_D(\RP^m,\XS;g)$  the subspace of 
$\Alg_D(\RP^m,\XS)$ defined by $\Alg^*_D(\RP^m,\XS ;g)=\Alg_D(\RP^m,\XS)\cap F(\RP^m,\XS;g)$.
%%%%%%%%
\par
%%%%
By the restriction, the map $\Gamma_D^{\p}$ induces maps
%%%%%%%%%%
%%%(1.24)%%%
\begin{equation}
%%%%%%%%%%%%
\begin{cases}
\Psi_D:A_D(m,\XS) \to \Alg^*_D(\RP^m,\XS),
\\
\Psi_D^{\p}:A_D(m,\XS;g) \to 
\Alg^*_D(\RP^m,\XS;g).
\end{cases}
\end{equation}
%%%%%%%%%%
%%%%%
Let 
%%%%%%%%
%%(1.25)%%
\begin{equation}
\begin{cases}
%j_{D}^{\C}:\Alg_D(\RP^m,\XS)\stackrel{\subset}{\rightarrow} \Map_D(\RP^m,\XS)
%\\
i_{D}^{\C}:\Alg_D^*(\RP^m,\XS)\stackrel{\subset}{\rightarrow} \Map_{D}^*(\RP^m,\XS)
\\
\hat{i}_{D}^{\C}:\Alg_D^*(\RP^m,\XS;g)\stackrel{\subset}{\rightarrow} F(\RP^m,\XS;g)\simeq \Omega^m\XS
\end{cases}
\end{equation}
%%%%
denote the inclusions.
It is easy to see that
%%%%%
%%(1.26)%%
\begin{equation}\label{jd}
%%%%%%%%
\begin{cases}
%j_D=j_{D}^{\C}\circ \Gamma_D:\widetilde{A_D}(m,\XS)\to \Map_{D}(\RP^m,\XS)
%\\
i_D=i_{D}^{\C}\circ \Psi_D:A_D(m,\XS)\to \Map_{D}^*(\RP^m,\XS)
\\
i_D^{\p}=\hat{i}_{D}^{\C}\circ \Psi_D^{\p}:A_D(m,\XS;g)\to F(\RP^m,\XS;g)
\simeq \Omega^m\XS
\end{cases}
\end{equation}
%%%%%%%%%%%%%%%%%
%%%%%%
%\par\vspace{2mm}\par
%%%%%
The contents of the above definitions are be summarized in the following 
diagram.
%%%%
%%%(Diagram)%%%
%%%%()%%%
\begin{equation*}
\xymatrix{%
\Map_{D}(\RP^m,\XS)     &     \Map^*_{D}(\RP^m,\XS) \ar@{_{(}->}[l]  &   F(\RP^m,\XS;g) \ar@{_{(}->}[l] 
\\
\Alg_D(\RP^m,\XS) \ar@{^{(}->}[u]^{j_{D}^{\C}}     &     \Alg_D^*(\RP^m,\XS) \ar@{^{(}->}[u]^{i_{D}^{\C}} \ar@{_{(}->}[l]  &   \Alg_D^*(\RP^m,\XS;g)  \ar@{^{(}->}[u]^{\hat{i}_{D}^{\C}} \ar@{_{(}->}[l]\\
\widetilde{A_D}(m,\XS) \ar@{->>}[u]^{\Gamma_D} \ar@/^4pc/[uu]^>>>>>>{j_D}   &     A_D(m,\XS) \ar@{->>}[u]^{\Psi_D} \ar@/^4pc/[uu]^>>>>>>{i_D} &   A_D(m,\XS;g) \ar@{->>}[u]^{\Psi_D^{\p}} \ar@/^4pc/[uu]^>>>>>>{i_D^{\p}}
}
\end{equation*}
%%%%%%%%%%%%%%%%%%%%%%%%%%%%%%%%%

%%%%%%%%%%%%%%%%%%%%%%%%%%%%%%
%%%(1.3 The main results)%%%%%
\subsection{The main results.}
%%%%%%%%%%%%%%%%%%%%%%%%%%%%%
%%%%
%%%%
%%%%
%%%%
For $d_k\in \Z_{\geq 1}$ $(1\leq k\leq r)$,
let $\db$ and $D(d_1,\cdots ,d_r;m)$
be the positive integers defined by 
%%%%%%%%%
%%(1.27)%%
\begin{equation}\label{Dnumbers}
%%%%%%%%%%
\db =\min \{d_1,d_2,\cdots ,d_r\},
\qquad
D(d_1,\cdots ,d_r;m)=(2r_{\rm min}-m-1)\db -2.
\end{equation}
%%%%%
%%
%%
Let $g\in \Alg_D^*(\RP^{m-1},\XS)$ be any fixed based algebraic map and 
we choose
an element $(g_1,\cdots ,g_r)\in A_D(m-1,\XS)$ such that
$g=i_D(g_1,\cdots ,g_r)=[g_1,\cdots ,g_r]$.
\par\vspace{2mm}\par
The main result  of this paper is the following theorem. 

%%%(Theorem 1.5)%%%
\begin{thm}\label{thm: I}
%%%%%%%%%%%%%%%%%
%%%%%%%%%%%%%%%%%
Let $\Sigma$ be a complete fan in $\R^n$
%be a normal fan of an $n$ dimensional simple $q_P$-neighborly
%lattice polytope $P$ in $\R^n$ 
satisfying the conditions
$(\ref{H(k)}.1)$, $(\ref{H(k)}.2)$, and
$\XS$ be a smooth compact toric variety associated to the fan
$\Sigma$.
If $2\leq m\leq 2(\rmin -1)$ and $D=(d_1,\cdots ,d_r)\in (\Z_{\geq 1})^r$, the map
$
i_D^{\p}:A_D(m,\XS;g) \to F(\RP^m,\XS;g)\simeq \Omega^m\XS
$
is a homology equivalence through dimension $D(d_1,\cdots ,d_r;m).$
%%%%
%%%%%%%%%%%%%%%
\end{thm}
%%%%%%%%%%%%%%%%%%%%%%%%%%
%%%%%(Remark 1.6)%%%%%
\begin{rmk}
{\rm
A map $f:X\to Y$ is called  {\it a homology equivalence through dimension} 
%(resp. {\it a homotopy equivalence through dimension}) 
$N$ 
if
%the induced homomorphism
$
f_*:H_k(X,\Z) \to H_k(Y,\Z)
$
is an isomorphism for any $k\leq N$.}
%A map $f:X\to Y$ is called  {\it a homology equivalence through dimension} 
%%(resp. {\it a homotopy equivalence through dimension}) 
%$N$ 
%if
%%the induced homomorphism
%$
%f_*:H_k(X,\Z) \to H_k(Y,\Z)
%$
%%$(\mbox{resp. }f_*:\pi_k(X)\to \pi_k(Y))$
%is an isomorphism for any $k\leq N$.
\end{rmk}
\par\vspace{2mm}\par

By the same method as in \cite[Theorem 3.5]{AKY1}
(cf. \cite{GM}), we also obtain the following:
%%
%%(Corollary 1.7)%%%%%%%%%
\begin{crl}\label{cor: II}
%%%%%%%%%%
Under the same assumptions as in Theorem \ref{thm: I}, 
if $2\leq m\leq 2(\rmin -1)$ and $D=(d_1,\cdots ,d_r)\in (\Z_{\geq 1})^r$,
the maps
$$
\begin{cases}
j_D:\widetilde{A_D}(m,\XS) \to \Map_D(\RP^m,\XS)
\\
i_D:A_D(m,\XS)\to \Map^*_D(\RP^m,\XS)
\end{cases}
$$
are  homology equivalences through dimension 
$D(d_1,\cdots ,d_r;m).$
\qed
%%%%%%%%%
\end{crl}
%%%%%%%%%%
%%
%%
%%%(Remark 1.8)%%%%%%%
\begin{rmk}\label{rmk: m=1}
{\rm
It is known that
the assertion of Corollary  \ref{cor: II} does not hold for $m=1$.
For example, we can easily see this for $\XS =\CP^n$
(cf. \cite[Theorem 3.5]{GKY2}).}
\end{rmk}
%%%%%%%%%

%%%%%%
Finally  we consider an example illustrating these results.
%{\rm
If $\XS =\CP^n$, %by using Example \ref{exa: projective space}
it is easy to see that
$(r,\rmin )=(n+1,n+1)$ ( this case was already treated in \cite{KY5}).
%%%
Now
consider the case  
$\XS =H(k)$
(the Hirzebruch surface). 
%%%%
Since there is an isomorphism $H(k)\cong H(-k)$,  we may assume 
%without loss of generality 
that $k\geq 0$.
%\par
Let ${\{\bf n}_j\in\Z^2\ \vert \ 1\leq j\leq 4\}$ be the set of the primitive elements of the fan 
$\Sigma$
given in Example \ref{exa: Hirzebruch}.
It is easily see that $(r,\rmin)=(4,2)$. 
%because $\Sigma =\Sigma_P$ and
%$P$ is a trapezoid.
Since $\sum_{j=1}^4d_j{\bf n}_j={\bf 0}$ if and only if
$(d_3,d_4)=(d_1,d_2+kd_1)$, we obtain the following result.
%%%
%%%
%%%%(Example 1.9)%%%%%%%
\begin{exa}[The case $(\XS ,m) =(H(k),2)$]
\label{exa: III}
%%%%%%%%%%%%%%%%%%%%%%%
If $k\geq 0$,
$d_j\geq 1$ $(j=1,2)$ are integers and
$D=(d_1,d_2,d_1,d_2+kd_1)$,  the  maps
$$
\begin{cases}
j_D:\widetilde{A_D}(2,H(k))\to \Map_D(\RP^2,H(k)),\quad
i_D:A_D(2,H(k))\to \Map_D^*(\RP^2,H(k)),
\\
i_D^{\p}:A_D(2,H(k);g)\to \Omega^2_DH(k)
\end{cases}
$$
are homology equivalences through dimension
$\min \{d_1,d_2\}-2$.
\qed 
%%%%%%%%%
\end{exa}
%%%%%%%%%%%%
%Let ${\bf n}_k\in\R^n$ ($0\leq k\leq n)$ denote the primitive elements of the fan $\Sigma$
%given in Example \ref{exa: projective space}.
%In this case, we can easily see that $(r,q_{P})=(n+1,n)$ because $\Sigma =\Sigma_P$
% and $P$ is an $(n-1)$-simplex.
%Moreover, because $\sum_{k=0}^nd_k{\bf n}_k={\bf 0}$ if and only if
%$d_0=d_1=\cdots =d_n$ and $\rmin =n+1$, we have the following result.
%%%%
%%%%%(Example 1.9)%%
%\begin{exa}%[The case $\XS =\CP^n$]
%\label{exa: IV}
%%%%%%
%Let $2\leq m\leq 2n$ and $d\geq 1$ be positive integers.
%If $D=(d,d,\cdots ,d)$,
%the maps
%$$
%\begin{cases}
%j_D:\widetilde{A_D}(m,\CP^n)\to \Map_D(\RP^2,\CP^n),\quad
%i_D:A_D(m,\CP^n)\to \Map_D^*(\RP^m,\CP^n),
%\\
%i_D^{\p}:A_D(m,\CP^n;g)\to \Omega^m_D\CP^n\simeq \Omega^mS^{2n+1}
%\end{cases}
%$$
%are homology equivalences through dimension
%$(2n-m+1)d-2$.
%\qed 
%%%%
%\end{exa}

%%%%%%%%%%%%%%%%%%%%%%%
%%%(SECTION 2)%%%%
\section{Simplicial resolutions.}\label{section 2}
%%%%%%%%%%%%%%%%%%

In this section, we summarize  the definitions of the non-degenerate simplicial resolution
and the associated truncated resolutions (\cite{AKY1}, \cite{KY4}, \cite{Mo2}, \cite{Mo3}, \cite{Va}).
%%%%%
%%(Definition 2.1)%%
\begin{dfn}\label{def: 2.1}
%%%%%%
{\rm
(i) For a finite set $\textbf{\textit{v}} \subset \R^N$,
let $\sigma (\textbf{\textit{v}})$ denote the convex hull spanned by 
$\textbf{\textit{v}}.$
%%%
%%%
Let $h:X\to Y$ be a surjective map such that
$h^{-1}(y)$ is a finite set for any $y\in Y$, and let
$i:X\to \R^N$ be an embedding.
Let  $\mathcal{X}^{\Delta}$  and $h^{\Delta}:{\mathcal{X}}^{\Delta}\to Y$ 
denote the space and the map
defined by
%%%
$$
\mathcal{X}^{\Delta}=
\big\{(y,u)\in Y\times \R^N:
u\in \sigma (i(h^{-1}(y)))
\big\}\subset Y\times \R^N,
\ h^{\Delta}(y,u)=y.
$$
The pair $(\mathcal{X}^{\Delta},h^{\Delta})$ is called
{\it a simplicial resolution of }$(h,i)$.
In particular, $(\mathcal{X}^{\Delta},h^{\Delta})$
is called {\it a non-degenerate simplicial resolution} if for each $y\in Y$
any $k$ points of $i(h^{-1}(y))$ span $(k-1)$-dimensional simplex of $\R^N$.
%%%%
\par
(ii)
For each $k\geq 0$, let $\mathcal{X}^{\Delta}_k\subset \mathcal{X}^{\Delta}$ be the subspace
given by 
$$
\mathcal{X}_k^{\Delta}=\big\{(y,u)\in \mathcal{X}^{\Delta}:
u \in\sigma (\textbf{\textit{v}}),
\textbf{\textit{v}}=\{v_1,\cdots ,v_l\}\subset i(h^{-1}(y)),l\leq k\big\}.
$$
%%%%
We make identification $X=\mathcal{X}^{\Delta}_1$ by identifying 
 $x\in X$ with the pair
$(h(x),i(x))\in \mathcal{X}^{\Delta}_1$,
and we note that  there is an increasing filtration
$$
\emptyset =
\mathcal{X}^{\Delta}_0\subset X=\mathcal{X}^{\Delta}_1\subset \mathcal{X}^{\Delta}_2\subset
\cdots \subset \mathcal{X}^{\Delta}_k\subset \mathcal{X}^{\Delta}_{k+1}\subset
\cdots \subset \bigcup_{k= 0}^{\infty}\mathcal{X}^{\Delta}_k=\mathcal{X}^{\Delta}.
$$
}
\end{dfn}
%%%%%

%%%(Lemma 2.2)%%
%\begin{lmm}[\cite{KY4}, \cite{Mo2}, \cite{Mo3}, \cite{Va}]\label{lemma: simp}
%%%%%%%%%
%Let $h:X\to Y$ be a surjective map such that
%$h^{-1}(y)$ is a finite set for any $y\in Y,$ and let
%$i:X\to \R^N$ be an embedding.
%%%
%\par
%%%(i)%%
%$\I$
%If $X$ and $Y$ are semi-algebraic spaces and the
%two maps $h$, $i$ are semi-algebraic maps, then
%$h^{\Delta}:\mathcal{X}^{\Delta}\stackrel{\simeq}{\rightarrow}Y$
%is a homotopy equivalence.
%%%%%
%\par
%$\II$
%If there is an embedding $j:X\to \R^M$ such that the associated simplicial resolution
%$(\tilde{\mathcal{X}}^{\Delta},\tilde{h}^{\Delta})$
%of $(h,j)$ is non-degenerate,
%the space $\tilde{\mathcal{X}}^{\Delta}$
%is uniquely determined up to homeomorphism.
%Moreover,
%there is a filtration preserving homotopy equivalence
%$q^{\Delta}:\tilde{\mathcal{X}}^{\Delta}\stackrel{\simeq}{\rightarrow}{\mathcal{X}}^{\Delta}$ such that $q^{\Delta}\vert X=\mbox{id}_X$.
%\qed
%%%%
%\end{lmm}
%%%%(End of Lemma 2.2)%%

%%%(Remark 2.2)%%%%%%%%%%%
%%
\begin{rmk}\label{Remark: non-degenerate}
{\rm
Even for a  surjective map $h:X\to Y$ which is not finite to one,  
it is still possible to construct an associated non-degenerate simplicial resolution.
In fact, a non-degenerate simplicial resolution may be
constructed by choosing a sequence of embeddings
$\{\tilde{i}_k:X\to \R^{N_k}\}_{k\geq 1}$ satisfying the following two conditions
for each $k\geq 1$ (cf. \cite{Va}).
%%%(Condition of non-degenerate simp. resolution)%%
\begin{enumerate}
\item[(\ref{def: 2.1}$)_k$]
%%%%%%%%%%%%%%%%%%%
\begin{enumerate}
%%(i)%%%
\item[(i)]
For any $y\in Y$,
any $t$ points of the set $\tilde{i}_k(h^{-1}(y))$ span $(t-1)$-dimensional affine subspace
of $\R^{N_k}$ if $t\leq 2k$.
%%(ii)%%%
\item[(ii)]
$N_k\leq N_{k+1}$ and if we identify $\R^{N_k}$ with a subspace of
$\R^{N_{k+1}}$, 
then $\tilde{i}_{k+1}=\hat{i}\circ \tilde{i}_k$,
where
$\hat{i}:\R^{N_k}\stackrel{\subset}{\rightarrow} \R^{N_{k+1}}$
denotes the inclusion.
\end{enumerate}
\end{enumerate}
%%%
Let
%%%()%%%%%
%begin{equation*}\label{2.1}
$\dis\mathcal{X}^{\Delta}_k=\big\{(y,u)\in Y\times \R^{N_k}:
u\in\sigma (\textbf{\textit{v}}),
\textbf{\textit{v}}
=\{v_1,\cdots ,v_l\}\subset \tilde{i}_k(h^{-1}(y)),l\leq k\big\}.$
%\end{equation*}
%%%%
%%
Then
by identifying  ${\cal X}^{\Delta}_k$ with a subspace
of ${\cal X}_{k+1}^{\Delta}$,  we define the non-degenerate simplicial
resolution ${\cal X}^{\Delta}$ of  $h$ as the union  
$\dis {\cal X}^{\Delta}=\bigcup_{k\geq 1} {\cal X}^{\Delta}_k$.
%Non-degenerate simplicial resolutions have a long been used in algebraic geometry, and play  the central role in the work of Vassiliev  \cite{Va}.
}
\end{rmk}

%%%
%\par\vspace{2mm}\par
%%%%%%%%

%In many practical cases the embedding  used to construct a simplicial resolution is given by an explicit map which carries geometric information about the corresponding filtration. Typically such an embedding gives rise to a simplicial resolution that is non-degenerate only in low dimensions. In some situations, such a degenerate resolution may provide more information about the homology of the resolved space than the non-degenerate one. This is why a degenerate resolution (defined by a Veronese-like embedding defined in the next section)  was used in \cite{Mo2} and \cite{AKY1}.  However, in \cite{Mo3} a modification of the non-degenerate resolution, called, the truncated (after a a certain term)
%simplicial resolution  was used to obtain results that are (in most dimensions) better than the one derived by means of the degenerate resolution. 

%%%(Definition 2.3)%%%
\begin{dfn}\label{def: 2.3}
{\rm
Let $h:X\to Y$ be a surjective semi-algebraic map between semi-algebraic spaces, 
$j:X\to \R^N$ be a semi-algebraic embedding, and let
$({\cal X}^{\Delta},h^{\Delta}:{\cal X}^{\Delta}\to Y)$
denote the associated non-degenerate  simplicial resolution of $h$.
%%%
\par
%%%
Let $k$ be a fixed positive integer and let
$h_k:{\cal X}^{\Delta}\to Y$ be the map
defined by the restriction
$h_k:=h^{\Delta}\vert {\cal X}^{\Delta}$.
%\par
The fibres of the map $h_k$ are $(k-1)$-skeleta of the fibres of $h^{\Delta}$ and, in general,  fail to be simplices over the subspace
$$
Y_k=\{y\in Y:h^{-1}(y)\mbox{ consists of more than $k$ points}\}.
$$
Let $Y(k)$ denote the closure of the subspace $Y_k$.
We modify the subspace ${\cal X}^{\Delta}_k$ so as to make the all
the fibres of $h_k$ contractible by adding to each fibre of $Y(k)$ a cone whose base
is this fibre.
We denote by $X^{\Delta}(k)$ this resulting space and by
$h^{\Delta}_k:X^{\Delta}(k)\to Y$ the natural extension of $h_k$.
%\par
%%%%%%%%%%%%
Following  \cite{Mo3}, we call the map $h^{\Delta}_k:X^{\Delta}(k)\to Y$
{\it the truncated $($after the $k$-th term$)$  simplicial resolution} of $Y$.
Note that there is a natural filtration
$$
\emptyset =X^{\Delta}_0\subset X^{\Delta}_1\subset
%X^{\Delta}_2\subset
\cdots 
\subset X^{\Delta}_l\subset X^{\Delta}_{l+1}\subset \cdots
\subset  X^{\Delta}_k\subset X^{\Delta}_{k+1}=
X^{\Delta}_{k+2}
%=X^{\Delta}_{k+3}
=\cdots =X^{\Delta}(k),
$$
where $X^{\Delta}_l={\cal X}^{\Delta}_l$ if $l\leq k$ and
$X^{\Delta}_l=X^{\Delta}(k)$ if $l>k$.
}
\end{dfn}
\section{Non-degenerate simplicial resolutions.
}
\label{section 3}
%%%(Definition for Z_d for A_d(m,n;g)%%
%%%
%%%
%%%

%%(Definition 3.1)%%%%%%%
\begin{dfn}\label{Def: 3.1}
{\rm
Fix a based algebraic map $g\in\Alg_D^*(\RP^{m-1},\XS)$ of degree $D$ together with a representation $(g_1,\ldots,g_r)\in A_D(m-1,\XS)$
such that $g=[g_1,\cdots ,g_r]$.
 Note that 
$A_D(m,\XS;g)$ is an open subspace of the affine space $B_D=B_1\times \cdots \times B_r$. 
\par\vspace{2mm}\par
%%(i)%%%%
(i)
 Let $N_D=\dim_{\C} B_D=\sum_{k=1}^r\binom{m+d_k-1}{m}$, and 
let $\Sigma_D\subset B_D$ denote the \emph{discriminant} of $A_D(m,\XS;g)$ in $B_D$, that is,
the complement
%%%%
\begin{eqnarray*}
\Sigma_D
&=&
B_D\setminus A_D(m,\XS;g)
\\
&=&
\{(f_1,\cdots ,f_r)\in B_D\ \vert \ 
(f_1(\textbf{\textit{x}}),\cdots ,f_r(\textbf{\textit{x}}))\in \ZS
\mbox{ for some }\textbf{\textit{x}}\in \R^{m+1}\setminus \{{\bf 0}\}\}.
\end{eqnarray*}
\par
%%(ii)%%
(ii)
Let  $Z_D\subset \Sigma_D\times \R^m$
denote %the subspace
{\it the tautological normalization} of 
 $\Sigma_D$
consisting of all pairs 
$(F,x)=((f_1,\ldots ,f_r),
(x_0,\ldots ,x_{m-1}))\in \Sigma_D\times\R^m$
such that 
$(f_1(\textbf{\textit{x}}),\cdots ,f_r(\textbf{\textit{x}}))\in \ZS$,
where we set
$\textbf{\textit{x}}=(x_0,\cdots ,x_{m-1},1).$
%%%
%\par
%%%%
Projection on the first factor  gives a surjective map
$\pi_D :Z_D\to\Sigma_D.$
%%%
%%%
}
\end{dfn}

Our goal in this section is to construct, by means of the
{\it non-degenerate} simplicial resolution  of the discriminant, a spectral sequence converging to the homology of the space
$A_D(m,\XS;g)$.
\begin{dfn}\label{non-degenerate simp.}
%%%
{\rm
Let 
$(\SZ,{\pi}^{\Delta}_D:\SZ\to\Sigma_D)$ 
%and
%$(\tilde{\SZ}(d),\ ^{\p}\tilde{\pi}_d^{\Delta}:\SZ(d)\to\Sigma_d^*)$ 
be the non-degenerate simplicial resolution of the surjective map
$\pi_D:Z_D\to \Sigma_D$ 
with the following natural increasing filtration as in Definition \ref{def: 2.1},
$
\SZ_0=\emptyset
\subset \SZ_1\subset 
\SZ_2\subset \cdots
\subset 
\SZ=\bigcup_{k= 0}^{\infty}\SZ_k.
$
}
\end{dfn}

%\begin{rem}
%{\rm
%Note that our notation conflicts with the one used in the previous section for the simplicial resolution truncated after the $d$-th term.  Now the $(d)$ in $\SZ(d)$ refers not to truncation but to the homogeneous degree of the polynomials in $A_d^*$ etc.  We shall continue with this abuse of notation when we define truncated resolutions below.
%}
%\end{rem}

%%%%%%
%By Lemma \ref{lemma: simp} 
By \cite[Lemma 2.2]{KY4}, the map
$\pi_D^{\Delta}:
\SZ\stackrel{\simeq}{\rightarrow}\Sigma_D$
is a homotopy equivalence,
which
extends to  a homotopy equivalence
%%%%%%%%%
$\pi_{d+}^{\Delta}:\SZ_+\stackrel{\simeq}{\rightarrow}{\Sigma_{D+}},$
%%%%%
where $X_+$ denotes the one-point compactification of a
locally compact space $X$.
%%
%\par
%%%
Since
${{\cal X}_k^{D}}_+/{\SZ_{k-1}}_+
\cong (\SZ_k\setminus \SZ_{k-1})_+$,
we have a spectral sequence 
%%%%%%%%%%%(3.1)%%%%%
\begin{equation}\label{SS-original}
%%%%
\big\{E_{t;D}^{k,s},
d_t:E_{t;D}^{k,s}\to E_{t;D}^{k+t,s+1-t}
\big\}
\Rightarrow
H^{k+s}_c(\Sigma_D,\Z),
\end{equation}
%%%
where
$E_{1;D}^{k,s}=H^{k+s}_c(\SZ_k\setminus\SZ_{k-1},\Z)$ and
$H_c^k(X,\Z)$ denotes the cohomology group with compact supports given by 
$
H_c^k(X,\Z)= H^k(X_+,\Z).
$
%%%%%%%%%%%%%
%\par\vspace{2mm}\par
%%%%%
\par
By Alexander duality  there is a natural
isomorphism
%%%(3.2)%%%
\begin{equation}\label{Al}
H_k(A_D(m,\XS;g),\Z)\cong
H_c^{2N_D-k-1}(\Sigma_D,\Z)
\quad
\mbox{for }1\leq k\leq 2N_D-2.
\end{equation}
%%%
By
reindexing we obtain a
spectral sequence
%%
%%%(3.3)%%%
\begin{eqnarray}\label{SS}
%%%%%%%%%%%%%%%%%%%
&&\big\{\E^{t;D}_{k,s}, \tilde{d}^{t}:\E^{t;D}_{k,s}\to \E^{t;D}_{k+t,s+t-1}
\big\}
\Rightarrow H_{s-k}(A_D(m,\XS;g),\Z)
\end{eqnarray}
%%%%%%%
if $s-k\leq 2N_D-2$,
where
$\E^{1;D}_{k,s}=
\tilde{H}^{2N_D+k-s-1}_c(\SZ_k\setminus\SZ_{k-1},\Z).$
%and $\tilde{E}^t_{r,s}(d)=E_t^{r,N_d^*-1-s}(d).$
%%%%%%
\par\vspace{2mm}\par
%%%%
For a connected space $X$, 
let $F(X,k)$ denote the configuration space  of distinct $k$ points in $X$.
The symmetric group $S_k$ of $k$ letters acts on $F(X,k)$ freely by permuting 
coordinates. Let $C_k(X)$ be the configuration space of unordered 
$k$-distinct  points in $X$ given by the orbit space
$C_k(X)=F(X,k)/S_k$.
%%%
Similarly, let
 $L_k\subset (\R^m\times \ZS)^k$ denote the subspace
defined by
%%%%%()
%\begin{equation*}\label{l}
$L_k=\{((x_1,s_1),\cdots ,(x_k,s_k))\ \vert \ 
x_j\in \R^m,s_j\in\ZS,
x_l\not= x_j\mbox{ if }l\not= j\}.$
%\end{equation*}
%%%%%%%%%
The group $S_k$ also acts on $L_k$ by permuting coordinates, and let
$C_k$ denote the orbit space
%%%%%
%%%%(3.4)%%%%
\begin{equation}\label{Ck}
C_k=L_k/S_k.
\end{equation}
%%%%%%%%%%%
%%%%
Note that $C_k$ is a cell complex of dimension
$(m+2r-2r_{\rm min})k$
(cf. (\ref{pri})).
%and that
%there is a homeomorphism
%
%%%%(3.4)%%%
%\begin{equation}\label{equ: Ck}
%%%%%%%%%%
%C_k\cong C_k(\R^m)\times (\ZS)^k.
%\end{equation}
%%%%%% 
%%%%
%%%%
%%%%(Lemma 3.3)%%%%
\begin{lmm}\label{lemma: vector bundle*}
%%%%%%%%%%%%%%%%%%
%%%
If  
$1\leq k\leq \db$,
$\SZ_k\setminus\SZ_{k-1}$
is homeomorphic to the total space of a real affine
bundle $\xi_{D,k}$ over $C_k$ with rank 
$l_{D,k}=2N_D-2kr+k-1$.
%%%%%%%%%%%%%%%%%%
\end{lmm}
%%%%%%%%(Proof of Lemma 3.3)%%%
\begin{proof}
%%%%%%%%%%%
The argument is exactly analogous to the one in the proof of  
\cite[Lemma 4.4]{AKY1}. Namely, an element of $\SZ_k\setminus\SZ_{k-1}$ is represented by $(F,u)=((f_1,\cdots ,f_r),u)$, where 
$F=(f_1,\cdots ,f_r)$ is a 
$r$-tuple of polynomials in $\Sigma_D$ and $u$ is an element of 
the interior of
the span of the images of $k$ distinct points $x_1,\cdots, x_k\in \R^m$ 
such that
%%%%%
$F(\textbf{\textit{x}}_j)=(f_1(\textbf{\textit{x}}_j),\cdots ,f_r(\textbf{\textit{x}}_j))\in \ZS$ for each $1\leq j\leq k$, 
%%%%
under a suitable embedding, where
we set $\textbf{\textit{x}}_j=(x_j,1)\in\R^{m+1}$.
%%%%
%%%%%%%%%%%%%
%%%()%%%%%
%%%%%%%%%%%%%
%\begin{equation}\label{pik}
%%%%%%%%%%%% 
Let $\pi_k :{\cal X}^{D}_k\setminus
{\cal X}^{D}_{k-1}\to C_k$
%\end{equation}
%%%%%%%%%%%%
be the projection map
$((f_1,\cdots ,f_r),u) \mapsto \{(x_1,F(\textbf{\textit{x}}_1)),\dots, (x_k,F(\textbf{\textit{x}}_k))\}$. 
(Note that the points $x_i$ are uniquely determined by $u$ by the construction of the non-degenerate simplicial resolution.)
\par
%%%%%(Fiber of pi_k)%%%%%%
Next, let $c=\{(x_j,s_j)\}_{j=1}^k\in C_k$
$(x_j\in \R^m$, $s_j\in \ZS)$ be any fixed element and consider the fibre  $\pi_k^{-1}(c)$.
%%%
For this purpose, for each $1\leq j\leq k$ let us consider the condition  
%%%(3.5)%%%
\begin{equation}\label{equ: pik}
%%%%%%%%%%%
F(\textbf{\textit{x}}_j)=(f_1(\textbf{\textit{x}}_j),\cdots ,f_r(\textbf{\textit{x}}_j))=s_j
\quad
\Leftrightarrow
\quad
f_t(\textbf{\textit{x}}_j)=s_{t,j}
\quad
\mbox{for }1\leq t\leq r,
\end{equation}
%%%%%%%%%%%
where we set $s_j=(s_{1,j},\cdots ,s_{r,j})$.
In general, %for each $1\leq t\leq r$,
the condition $f_t(\textbf{\textit{x}}_j)=s_{t,j}$ gives
one  linear condition on the coefficients of $f_t$,
which determines an affine hyperplane in $B_t$. 
%%%
Since $\{\textbf{\textit{x}}_j\}_{j=1}^k$ is mutually distinct, 
if $1\leq k\leq \db$,
by \cite[Lemma 4.3]{AKY1} the condition (\ref{equ: pik}) 
%that $f_t(\textbf{\textit{x}}_j)=s_{t,j}$ ($1\leq j\leq k$)
produces exactly $k$ independent conditions on the coefficients of $f_t$.
Thus  the space of polynomials $f_t$ in $B_t$ which satisfies
(\ref{equ: pik})
is the intersection of $k$ affine hyperplanes in general position
and it has codimension $k$ in $B_t$.
%%%%%
Hence, if $1\leq k\leq \db$
the fibre $\pi_k^{-1}(c)$ is homeomorphic  to the product of open $(k-1)$-simplex
 with the real affine space of dimension
 $2\sum_{t=1}^r\big(\binom{d_t+m-1}{m}-k\big)=2N_D-2kr.$
Thus $\pi_k$ is a real affine bundle over $C_k$ of rank $l_{D,k}
=2N_D-2kr+k-1$.
\end{proof}
%%(End of proof of Lemma 3.3)%%%

%%%%%%(Lemma 3.4)%%
\begin{lmm}\label{lemma: E11}
%%%%%%
If $1\leq k\leq  \db$, there is a natural isomorphism
$$
\E^{1;D}_{k,s}\cong
%H_c^{(n+1)r-s}(C_r(X_K),{\cal O}_r)
\tilde{H}_c^{2rk-s}(C_k,\pm \Z).
%H_{s-(2r_{\rm min}-m)k}(C_k,\pm \Z).
$$
%where the twisted coefficients system $\pm \Z$ on the space $C_k$ 
%is induced by the
%sign representation of the symmetric group.%\footnote{%
%This is explained in \cite[page 114 and 254]{Va})}
%the Thom isomorphism.%\footnote{%
%%%%%%%%%%%%%%%%%%%%%%%%%
%%%%%(FootNote 9)%%%%%%%%
%The twisted coefficients system $\pm \Z$ on
%$C_k$ is induced by the sign representation of the symmetric group. 
%(cf. (\ref{Ck}) and \cite[page 114 and 254]{Va}).}
\end{lmm}
%%%%
\begin{proof}
%%%(Proof of Lemma 3.4)%%
Suppose that $1\leq k\leq \db$.
By Lemma \ref{lemma: vector bundle*}, there is a
homeomorphism
$(\SZ_k\setminus\SZ_{k-1})_+\cong T(\xi_{D,k}),$
where $T(\xi_{D,k})$ denotes the Thom space of
$\xi_{D,k}$.
%%%%
%%%%%%%
Since 
$$
(2N_D+k-s-1)-l_{D,k}
=(2N_D+k-s-1)-(2N_D-2kr+k-1)
=
2rk-s,
$$
by using the Thom isomorphism theorem we obtain
a natural isomorphism
$$
\E^{1;D}_{k,s}
\cong \tilde{H}^{2N_D+k-s-1}(T(\xi_{D,k}),\Z)
\cong
\tilde{H}_c^{2rk-s}(C_k,\pm \Z),
$$
where the twisted coefficient system $\pm \Z$
appears in the computation of the cohomology of the Thom space
induced by the
sign representation of the symmetric group
as in  
\cite[page 37-38, page 114 and 254]{Va}).
This completes the proof.
%%%%%
\end{proof}
%%%(End of proof of Lemma 3.4)%%%%%%
%%
%%%%
%%%(SECTION 4)%%%%%%%%%%
\section{Truncated spectral sequences.}\label{section 4}
%%%%%%%%%%%%%%%%%%%%%%%%
%%%
%%%
%%%
%%%
%%%%%%%%%%%%%%%%%%%%%%%%%%%%%%%%%%%%
In this section, we prove a  key result (Theorem \ref{thm: sd}) 
about the homology stability of \lq\lq stabilization 
maps\rq\rq \ $s_D:A_D(m,\XS;g)\to A_{D+\textbf{\textit{a}}}(m,\XS;g)$.

%%%%%%(Definition 4.1)%%%%
\begin{dfn}\label{dfn: 4.1}
%%%%
{\rm
Let $X^{\Delta}$ denote the  truncated (after $\db$-th term) simplicial resolution 
of $\Sigma_D$ with its natural filtration as in Definition \ref{def: 2.3},
$$
\emptyset =X^{\Delta}_0\subset X^{\Delta}_1\subset
%X^{\Delta}(d)_2\subset
\cdots \subset  X^{\Delta}_{\db}\subset 
X^{\Delta}_{\db+1}=
X^{\Delta}_{\db+2}=
\cdots
 =X^{\Delta},
$$
where $X^{\Delta}_k=\SZ_k$ if $k\leq \db$ and
$X^{\Delta}_k=X^{\Delta}$ if $k\geq \db+1$.
}
\end{dfn}

%%%(Remark 4.2)%%%%%%%%%%%%
\begin{rmk}\label{rmk: truncated}
%%%%%%%%%%%%%%%%%%%%%%%%%%%
{\rm
Note that our notation  
$X^{\Delta}$ conflicts with that of  \cite{Mo3} and Definition 2.3, where
${\cal X}^{\Delta}$ denotes the non-degenerate simplicial resolution. 
}
\end{rmk}
%%%%
%%%%
%\par\vspace{2mm}\par
%%%%%
%By Lemma \ref{Lemma: truncated} 
By \cite[Lemma 2.5]{KY4},
there is a homotopy equivalence
$\pi^{\Delta}:X^{\Delta} \stackrel{\simeq}{\rightarrow}
\Sigma_D$.
Hence, %by replacing $X^{\Delta}_k$ instead of ${\cal X}^{D}_k$
%in (\ref{SS-original}),
by using the filtration on $X^{\Delta}$ given in Definition \ref{dfn: 4.1}, 
we have a spectral sequence
%%%%%(4.1)%%%%
\begin{equation}\label{SS-*}
%%%%
\big\{\hat{E}_{t;D}^{k,s},
d_t:\hat{E}_{t;D}^{k,s}\to \hat{E}_{t;D}^{k+t,s+1-t}
\big\}
\Rightarrow
H^{k+s}_c(\Sigma_D,\Z),
\end{equation}
%%%
where
$\hat{E}_{1;D}^{k,s}=\tilde{H}^{k+s}_c(X^{\Delta}_k\setminus X^{\Delta}_{k-1},\Z)$.
Then by reindexing and using Alexander duality
(\ref{Al}), %as in
%(\ref{SS}),
we obtain the truncated spectral sequence
(cf. \cite[(4.1)]{KY4}, \cite{Mo3})
%%
%%%(4.2)%%%
\begin{eqnarray}\label{SSS}
%%%%%%%%%%%%%%%%%%%
&&\big\{E^t_{k,s}, d^t:E^t_{k,s}\to E^t_{k+t,s+t-1}\big\}
\Rightarrow H_{s-k}(A_D(m,\XS;g),\Z)
\end{eqnarray}
%%%%%%%
if $s-k\leq 2N_D-2$,
where
$E^1_{k,s}=
\tilde{H}^{2N_D+k-s-1}_c(X^{\Delta}_k\setminus X^{\Delta}_{k-1},\Z).$
%and $\tilde{E}^t_{r,s}(d)=E_t^{r,N_d^*-1-s}(d).$
%%%%%%

%%%%%%(Lemma 4.3)%%
\begin{lmm}\label{lemma: E1}
%%%%%%
\begin{enumerate}
\item[$\I$]
If $1\leq k\leq  \db$, there is a natural isomorphism
$$
E^1_{k,s}\cong
%H_c^{(n+1)r-s}(C_r(X_K),{\cal O}_r)
\tilde{H}_c^{2rk-s}(C_k,\pm \Z).
%H_{s-(2r_{\rm min}-m)k}(C_k,\pm \Z).
$$
%%(ii)%%
\item[$\II$] 
$E^1_{k,s}=0$ if $k<0$, or if $k\geq \db+2$, or if $k=0$ and $s\not= 2N_D-1$.
%%
%%(iii)%%%%%
\item[$\III$] If $1\leq k\leq \db$,
$E^1_{k,s}=0$ for any $s\leq (2r_{\rm min}-m)k-1$.
%%(iv)%%
\item[$\IV$] If $k=\db+1$,
$E^1_{\db+2,s}=0$ for any $s\leq (2r_{\rm min}-m)\db-1$.
\end{enumerate}
%Here the meaning of  $(\pm \Z)^{\otimes (n-m)}$  is the same as in {\rm \cite{Va}}.
\end{lmm}
%%%%
%%%(Proof of Lemma 4.3)%%
\begin{proof}
%%%%%%
Since 
$X^{\Delta}_k\setminus X^{\Delta}_{k-1}
=\SZ_k\setminus\SZ_{k-1}$ for $1\leq k\leq \db$, the assertion (i) follows from
Lemma \ref{lemma: E11}. 
%%%%%%%%
Next, because $\XD_0=\emptyset$ and $X^{\Delta}=\XD_k$ for $k\geq \db+2$,
the assertion (ii) easily follows.
%\par
%%%(iii)%%%%
Now  we assume that $1\leq k\leq \db$ and try to prove (iii).
Since $\dim C_k=(m+2r-2\rmin)k$,
$2rk-s>\dim C_k$ iff
%
%Because $s-(2r_{\rm min}-m)k<0$ if and only if
$s\leq (2r_{\rm min}-m)k-1$ and (iii)  follows from (i).
%\par
It remains to show (iv).
%%%%%%
An easy computation shows that
\begin{eqnarray*}
\dim (
\SZ_{k}\setminus \SZ_{k-1})&=&
(2N_D-2kr+k-1)+\dim C_k
\\
&=&(2N_D-2kr+k-1)+(m+2r-2r_{\rm min})k
\\
&=&
2N_D-(2r_{\rm min}-m-1)k-1.
%%%%%%%%%%
\end{eqnarray*}
%%%%%%%
Since
$\dim (\XD_{\db+1}\setminus \XD_{\db})=
\dim (
\SZ_{\db}\setminus \SZ_{\db -1})+1$
by \cite[Lemma 2.6]{KY4}, 
$$
\dim (\XD_{\db+1}\setminus \XD_{\db})
=2N_D-(2r_{\rm min}-m-1)\db.
$$
%%%%
Since
$
E^1_{\db+1,s}=\tilde{H}_c^{2N_D+\db-s}(\XD_{\db +1}\setminus \XD_{\db},\Z)
$
and 
$$
2N_D+\db-s>\dim (\XD_{\db+2}\setminus \XD_{\db+1})\  
\Leftrightarrow
\ s\leq (2r_{\rm min}-m)\db -1,
$$
 we see that
$E^1_{\db+1,s}=0$ for $s\leq (2r_{\rm min}-m)\db-1$.
%%%%%%%%%%%%%%%%
\end{proof}
%%%%(End of proof of Lemma 4.3)%%%
%%%%%
%%%%%
%%%%%
%%(Truncated simplicial resolution of \SZ (d+2))%%%%%
%%%%
%%%%
%%%%
Let $\textbf{\textit{a}}=(a_1,\cdots ,a_r)\in (\Z_{\geq 1})^r$
be fixed $r$-tuple of positive integers such that
%%%(4.3 )%%%
\begin{equation}\label{number: a}
%%%%%%%%
\sum_{k=1}^ra_k{\bf n}_k={\bf 0}.
\end{equation}
%%%%%%%%
We set
$D+\textbf{\textit{a}}=(d_1+a_1,\cdots ,d_r+a_r)$.
Similarly,
let $Y^{\Delta}$ denote the  truncated (after $\db$-th term)
simplicial resolution 
of $\Sigma_{D+\textbf{\textit{a}}}$ with its natural filtration
$$
\emptyset =
Y^{\Delta}_0\subset Y^{\Delta}_1\subset
%X^{\Delta}(d)_2\subset
\cdots \subset  Y^{\Delta}_{\db}\subset Y^{\Delta}_{\db+1}=
Y^{\Delta}_{\db+2}=
\cdots
 =Y^{\Delta},
$$
where $Y^{\Delta}_k=\SZd_k$ if $k\leq \db$ and
$Y^{\Delta}_k=Y^{\Delta}$ if $k\geq \db+1$.
\par
%%%%%
By \cite[Lemma 2.5]{KY4}, 
there is a homotopy equivalence
$^{\p}\pi^{\Delta}:Y^{\Delta} \stackrel{\simeq}{\rightarrow}
\Sigma_{D+\textbf{\textit{a}}}$.
Hence, by using the same method as above,
we obtain a spectral sequence
%%%%%
%%%(4.4)%%%
%%%%%%%%%%%%%%%%%%%%
\begin{eqnarray}\label{SSSS}
%%%%%%%%%%%%%%%%%%%
&&\big\{\ ^{\p}E^t_{k,s},\  ^{\p}d^t:E^t_{k,s}\to 
\ ^{\p}E^t_{k+t,s+t-1}\big\}
\Rightarrow H_{s-k}(A_{D+\textbf{\textit{a}}}(m,\XS;g),\Z)
\end{eqnarray}
%%%%%%%
if $s-k\leq 2N_{D+\textbf{\textit{a}}}-2$,
where
$^{\p}E^1_{k,s}=
\tilde{H}^{2N_{D+\textbf{\textit{a}}}+k-s-1}_c(Y^{\Delta}_k\setminus Y^{\Delta}_{k-1},\Z).$
%and $\tilde{E}^t_{r,s}(d)=E_t^{r,N_d^*-1-s}(d).$
\par
Applying again the same argument  we obtain the following
result.

%%%%
%%%%%%%%%%%%%%%%%%%%%%%
%%%(Lemma 4.4)%%%
\begin{lmm}\label{lemma: E1*}
%%%%%%%%%%%%%%%%%%%%%%
\begin{enumerate}
\item[$\I$]
If $1\leq k\leq  \db$, there is a natural isomorphism
$$
^{\p}E^1_{k,s}\cong
%H_c^{(n+1)r-s}(C_r(X_K),{\cal O}_r)
\tilde{H}_c^{2rk-s}(C_k,\pm \Z).
%H_{s-(2r_{\rm min}-m)k}(C_k,\pm \Z).
$$
%%%(ii)%%%
\item[$\II$]
$^{\p}E^1_{k,s}=0$ if $k<0$, or if $k\geq \db+2$, or if $k=0$ and 
$s\not= 2N_{D+\textbf{\textit{a}}}-1$.
%%
%%(iii)%%%%%
\item[$\III$]
%%%%%%%% 
If $1\leq k\leq \db$,
$^{\p}E^1_{k,s}=0$ for any $s\leq (2r_{\rm min}-m)k-1$.
%%%(iv)%%%
\item[$\IV$]
%%%
If $k=\db +1$,
$^{\p}E^1_{\db+2,s}=0$ for any $s\leq (2r_{\rm min}-m)\db -1$.
\qed
\end{enumerate}
\end{lmm}
%%%%%

%%%%
%%%%%%%%%%%%%%%%%%%%%%%%%
%%%(Definition 4.5)%%%%%%
\begin{dfn}\label{Def: 4.4}
%%%%%%%%%%%%%%%%%%%%%%%%%
{\rm
Let $g\in \Alg_D^*(\RP^{m-1},\XS)$ be a fixed algebraic map,
and let $(g_1,\cdots ,g_r)\in A_D(m-1,\XS)$ be its fixed representative.
Let
 $\textbf{\textit{a}}=(a_1,\cdots,a_r)\in (\Z_{\geq 1})^r$ be a fixed $r$-tuple of positive integers satisfying the condition (\ref{number: a}).
%%% 
If we set $\tilde{g}=\sum_{k=0}^mz_k^2,$ we see that the tuple
$((\tilde{g}\vert_{z_m=0})^{a_1}g_1,\cdots ,(\tilde{g}\vert_{z_m=0})^{a_r}g_r)$ can also be chosen as a representative of the
map $g\in \Alg_{D+\textbf{\textit{a}}}^*(\RP^{m-1},\XS)$.
So one can define a stabilization map
$s_D:A_D(m,\XS;g)\to A_{D+\textbf{\textit{a}}}(m,\XS;g)$ by
%%%%%%%%
%%(4.5)%%%%%%%%%%%%%%%%%%%
\begin{equation}\label{sd}
%%%%%%%%%%%%%%%%%%%%%%%%%
s_D(f_1,\cdots ,f_r)=(\tilde{g}^{a_1}f_1,\cdots ,\tilde{g}^{a_r}f_r)
\quad 
\mbox{ for }\quad (f_1,\cdots ,f_r)\in A_D(m,\XS;g).
%s_D:A_D(m,\XS;g)\to A_{D+\textbf{\textit{a}}}(m,\XS;g)
\end{equation}
%%%
Because there is a commutative diagram
%%%%%%%%%%%
%%%(4.6)%%%
\begin{equation}\label{diagram_sd}
\begin{CD}
A_D(m,\XS;g) @>s_D>> A_{D+\textbf{\textit{a}}}(m,\XS;g)
\\
@V{i_D^{\p}}VV @V{i_{D+\textbf{\textit{a}}}^{\p}}VV
\\
F(\RP^m,\XS;g) @>=>> F(\RP^m,\XS;g)
\end{CD}
\end{equation}
it induces a map
%%%(4.7)%%
\begin{equation}\label{sdinfty}
%%%%%%%
s_{D,\infty}=\lim_{k\to\infty} s_{D+k\textbf{\textit{a}}}:
A_{D,\infty}(m,\XS;g)%=\lim_{k\to \infty}A_{d+2k}(m,n;g)
\to F(\RP^m,\XS;g)\simeq\Omega^m_0\XS,
\end{equation}
%%%%%%%
where 
$A_{D,\infty}(m,\XS;g)$ denotes
the colimit 
$\dis \lim_{k\to \infty}A_{D+k\textbf{\textit{a}}}(m,\XS;g)$
induced from the stabilization maps 
$s_{D+k\textbf{\textit{a}}}$'s $(k\geq 0)$.}
%%%%%%%
\end{dfn}
%%%%(End of Definition)%%%%%%%%
%%
%%(Theorem 4.6)%%%%%%%
\begin{thm}\label{thm: AKY1-stable}
%%%%%%%%%%%%%%%%%%%%%%
If $2\leq m\leq 2(\rmin -1)$, the map
$\dis s_{D,\infty}:A_{D,\infty}(m,\XS;g)
\stackrel{\simeq}{\rightarrow} \Omega^m_0\XS$
is a homology equivalence.
\end{thm}
%%%%%
%%%
We postpone the proof of Theorem \ref{thm: AKY1-stable}
 to \S \ref{section: main result}.  We first prove another key  result (Theorem \ref{thm: sd}).
%%%
%\par
%%%%%%%%%%%%%
Recall the %stabilization 
definition of the map
$s_D$, and
consider the map 
$\tilde{s}_D:\Sigma_D\to \Sigma_{D+\textbf{\textit{a}}}$ given by the multiplication,
$\tilde{s}_D(f_1,\cdots ,f_r)=(\tilde{g}^{a_1}f_1,\cdots ,\tilde{g}^{a_r}f_r)$.
One can show that
it
extends to the embedding 
$\tilde{s}_d:\R^{2(N_{D+\textbf{\textit{a}}}-N_D)}\times \Sigma_D\to
\Sigma_{D+\textbf{\textit{a}}}$
%by using %the method given in
 and that
%By the universality of non-degenerate resolutions, 
it induces
the filtration preserving open embedding
$\hat{s}_D:\R^{2(N_{D+\textbf{\textit{a}}}-N_D)}\times\SZ\to \SZd$
by using \cite[Prop. 7]{MV} or 
\cite[page 103-106]{Va}. 
%between non-degenerate resolutions.
Hence, it also induces the filtration preserving open embedding
$\hat{s}_D:\R^{2(N_{D+\textbf{\textit{a}}}-N_D)}\times\XD  \to Y^{\Delta}$. 
Since one-point compactification is contravariant for open embeddings, it induces a homomorphism 
of spectral sequences
%%%(4.8)%%%%
\begin{equation}
%%%%%%%%%%%%%%
\{\theta^t_{k,s}:E^t_{k,s}\to \ ^{\p}E^t_{k,s}\}.
\end{equation}

%%(Lemma 4.7)%%%%
\begin{lmm}\label{lemma: Thom}
%%%%%%%%
If $1\leq k\leq \db$,
$\theta^1_{k,s}:E^1_{k,s}\to \ ^{\p}E^1_{k,s}$
is an isomorphism for any $s$.
%%%%%%%%
\end{lmm}
%%%%%%%%%%%%%%%%
\begin{proof}
%%%%%%%%%%%%%%
Suppose that $1\leq r\leq \db$.
Then
it follows from the proof of Lemma \ref{lemma: vector bundle*}
that there is a homotopy commutative diagram of open affine bundles
$$
\begin{CD}
\SZ_k\setminus \SZ_{k-1} @>\pi_k>> 
C_k
\\
@V{\hat{s}_D}VV  \Vert @.
\\
\SZd_k\setminus \SZd_{k-1} @>{\pi_k}>> C_k
\end{CD}
$$
Since $X^{\Delta}_k\setminus X^{\Delta}_{k-1}=
\SZ_k\setminus \SZ_{k-1}$ and
$Y^{\Delta}_k\setminus Y^{\Delta}_{k-1}=
\SZd_k\setminus \SZd_{k-1}$,
by  Lemma \ref{lemma: E1} and Lemma \ref{lemma: E1*},
%and the naturality of the Thom isomorphism,
we have a commutative diagram
%%%()%%
\begin{equation*}\label{Thom}
%%%%%%%
\begin{CD}
E^1_{k,s}
@>T>\cong> \tilde{H}_c^{2rk-s}(C_k,\pm \Z)
%H_{s-k(2r_{\rm min}-m)}(C_k,\pm \Z)
\\
@V{\theta}^1_{k,s}VV  \Vert @.
\\
^{\p}E^1_{k,s}
@>T>\cong> \tilde{H}_c^{2rk-s}(C_k,\pm \Z)
%H_{s-k(2r_{\rm min}-m)}(C_k,\pm \Z )
\end{CD}
\end{equation*}
%%%%%
where $T$ denotes the Thom isomorphism.
Hence, 
$\theta^1_{k,s}$ is an isomorphism.
\end{proof}
%%%%%(End of proof of Lemma 4.7)%%%%%

%%%(Theorem 4.8)%%%
\begin{thm}\label{thm: sd}
%If $2\leq m\leq 2(\rmin -1)$,
%the stabilization map 
$s_D:A_D(m,\XS;g) \to A_{D+\textbf{\textit{a}}}(m,\XS;g)$ is a homology equivalence through
dimension $D(d_1,\cdots ,d_r;m)$.
\end{thm}
%%%%
%%%(Proof of Theorem 4.8)%%%%
\begin{proof}
%%%%%%%%%%
We set
$D_0=D(d_1,\cdots ,d_r;m)=(2\rmin -m-1)\db -2$, and
consider 
two spectral sequences
$$
\begin{cases}
%%%%%%
\{E^{t}_{k,s},d^t:E^t_{k,s}\to \ E^t_{k+t,s+t-1}\}\ & 
\Rightarrow \quad H_{s-k}(A_D(m,\XS;g),\Z),
\\
\{\ ^{\p}E^{t}_{k,s},\ ^{\p}d^t:\ ^{\p}E^t_{k,s}
\to \ ^{\p}E^t_{k+t,s+t-1}\}  & 
\Rightarrow \quad H_{s-k}(A_{D+\textbf{\textit{a}}}(m,\XS;g),\Z),
\end{cases}
$$
with a homomorphism 
$\{\theta^t_{k,s}:E^t_{k,s}\to \ ^{\p}E^t_{k,s}\}$
of spectral sequences.
\par
%%%
Next, 
we shall try to estimate the maximal positive integer $D_{\rm max}$ such that
$$
D_{\rm max}=\max\{N\in\Z:\theta^{\infty}_{k,s}
\mbox{ is  an isomorphism for all $(k,s)$ if }s-k\leq N\}.
$$
%%(The case r<0 or r \leq d+2)%%%%%
By Lemma \ref{lemma: E1} and  \ref{lemma: E1*}, we see that
$E^1_{k,s}=\ ^{\p}E^1_{k,s}=0$ if
$k<0$, or if $k>\db +1$, or if $k=\db +1$ with $s\leq (2\rmin -m)\db -1$.
Since $(2\rmin -m)\db -(\db +1)=D_0+1$,  we deduce that:
%()%%
%%%%%%%
%%%%%%%
\begin{enumerate}
%%%%%%%
\item[$(*)_1$]
if $k<0$ or $k\geq \db +1$,
$\theta^{\infty}_{k,s}$ is an isomorphism for all $(k,s)$ if
$s-k\leq D_0$.
\end{enumerate}
%\par
%%
Next, we assume that $0\leq k\leq \db$, and again investigate the condition for
$\theta^{\infty}_{k,s}$  to be an isomorphism.
%%%
%Then
Note that the group $E^1_{k_1,s_1}$ is not known for
%%(S_1)%%%
$(k_1,s_1)\in{\cal S}_1=\{(\db +1,s)\in\Z^2:s\geq (2\rmin -m)\db \}$.
%%%%
By considering the differentials
$d^1:E^1_{k,s}\to E^{1}_{k+1,s}$
and
$\ ^{\p}d^1:\ ^{\p}E^1_{k,s}\to \ ^{\p}E^{1}_{k+1,s}$
and applying Lemma \ref{lemma: Thom}, we see that
$\theta^2_{k,s}$ is an isomorphism if
$(k,s)\notin {\cal S}_1\cup {\cal S}_2$, where
%%(S_2)%%%
$$
{\cal S}_2=:
\{(k_1,s_1)\in\Z^2:(k_1+1,s_1)\in {\cal S}_1\}
=\{(\db ,s_1)\in \Z^2:s_1\geq (2\rmin -m)\db \}.
$$
%%%%
%%%%
A similar argument for the differentials $d^2$ and $^{\p}d^2$ shows that
$\theta^3_{k,s}$ is an isomorphism if
%%(S_3)%%
$(k,s)\notin \bigcup_{u=1}^3{\cal S}_u$, where
${\cal S}_3=\{(k_1,s_1)\in\Z^2:(k_1+2,s_1+1)\in {\cal S}_1\cup
{\cal S}_2\}.$
%%%%%
%\par
%%%%
Continuing in the same fashion,
considering the differentials
$d^t:E^t_{k,s}\to E^{t}_{k+t,s+t-1}$
and
$\ ^{\p}d^t:\ ^{\p}E^t_{k,s}\to \ ^{\p}E^{t}_{k+t,s+t-1},$
and applying 
Lemma \ref{lemma: Thom}, we  easily see that $\theta^{\infty}_{k,s}$ is an isomorphism
if $\dis (k,s)\notin {\cal S}:=\bigcup_{t\geq 1}{\cal S}_t
=\bigcup_{t\geq 1}A_t$,
where  $A_t$ denotes the set given by
%%%%(Def. of A_t)%%%%%%%
$$
A_t:=
\left\{
\begin{array}{c|l}
 &\mbox{ There are positive integers }l_1,l_2,\cdots ,l_t
\mbox{ such that},
\\
(k_1,s_1)\in \Z^2 &\  1\leq l_1<l_2<\cdots <l_t,\ 
k_1+\sum_{j=1}^tl_j=\db +1,
\\
& \ s_1+\sum_{j=1}^t(l_j-1)\geq (2\rmin -m)\db
\end{array}
\right\}.
$$
%%%%%%%%  
If $\dis A_t\not= \emptyset$, it is easy to see that
%%%
\begin{eqnarray*}
a(t)&=&\min \{s-k:(k,s)\in A_t\}=
(2\rmin -m)\db -(\db +1)+t
%\\
%&=&(n-m)(d+1)-1+t
=D_0+t+1.
\end{eqnarray*}
%%%
Hence, 
$\min \{a(t):t\geq 1,A_t\not=\emptyset\}=D_0+2,$
and we have the following:
%%(*2)%%%
\begin{enumerate}
%%%%%%%%
\item[$(*)_2$]
If $0\leq k\leq \db $,
$\theta^{\infty}_{k,s}$ is  an isomorphism for any $(k,s)$ if
$s-k\leq  D_0+1.$
\end{enumerate}
%%%%%%
Then, by $(*)_1$ and $(*)_2$, we see that
$\theta^{\infty}_{k,s}:E^{\infty}_{k,s}\stackrel{\cong}{\rightarrow}
\ ^{\p}E^{\infty}_{k,s}$ is an isomorphism for any $(k,s)$
if $s-k\leq D_0$.
%By  the Comparison Theorem for spectral sequences
Hence, we can now see that
$s_D$ is a homology equivalence through dimension
$D_0$.
\end{proof}
%%%%%%%%%%%%(End of Proof of Theorem 4.8)%%%%%%%
%%%%%
%%
%%
%%
%%
%%
%%%%%%%%%%%%%%%(SECTION 5)%%%
\section{The complement $\C^r\setminus \ZS$.}\label{section: polyhedral product}
%%%%%%%%%%%%%%%%%%%%%%%%%%%%%
%%
%%
%%
%%
%%
%%%%%%%%%%%%%%%%%%
In this section, we recall the basic results on polyhedral products and
investigate the connectivity of the complement $\C^r\setminus \ZS$.
 
%%(Definition 5.1)%%%%%%
\begin{dfn}
%%%%%
{\rm
Let $[r]=\{1,2,\cdots ,r\}$ be a set of indices and let
$K$ be a simplicial  complex on the vertex set $[r]$.
\par
%%()%%%%%
(i)
For each $\sigma =\{i_1,\cdots ,i_k\}\subset [r]$, let
$L_{\sigma}\subset \C^r$ denote {\it the coordinate subspace} in $\C^r$ given by
$L_{\sigma}=\{\textbf{\textit{x}}=(x_1,\cdots ,x_r)\in \C^r:x_{i_1}=\cdots =x_{i_k}=0\}$,
and $U(K)$ the complement  of the {\it  arrangement of coordinate subspaces} in $\C^r$

%%%(5.1)%%%
\begin{equation}\label{arrangement}
U(K)=\C^r\setminus \bigcup_{\sigma\notin K}L_{\sigma}.
\end{equation}
%%%%%%%
\par (ii)
Let $(\underline{X},\underline{A})$ be a collection of spaces of pairs
$\{(X_k,A_k)\}_{k=1}^r$.
The {\it  polyhedral product}
${\cal Z}_K(\underline{X},\underline{A})$ of the collection $(\underline{X},\underline{A})$ with respect to $K$
by
%%%(5.2)%%%
\begin{eqnarray}\label{polyhedral product}
%%%%%%%%%
{\cal Z}_K(\underline{X},\underline{A})&=&
\bigcup_{\sigma\in K}(\underline{X},\underline{A})^{\sigma},
\end{eqnarray}
%%%%%%%%%%
where we set
$(\underline{X},\underline{A})^{\sigma}=
\{(x_1,\cdots ,x_r)\in X_1\times \cdots \times X_r\ \vert \ 
x_j\in A_j\mbox{ if }j\notin \sigma\}$ for $\sigma \in K$.
%%%
\par
In particular,
when $(X_j,A_j)=(X,A)$ for each $1\leq j\leq r$, we set
${\cal Z}_K(\underline{X},\underline{A})={\cal Z}_K(X,A)$.
For $(X,A)=(D^2,S^1)$, ${\cal Z}_K={\cal Z}_K(D^2,S^1)$  is called the
{\it moment-angle complex} of type $K$, while
$DJ(K)={\cal Z}_K(\CP^{\infty},*)$  is called the {\it Davis-Januszkiwicz space} of type $K$.
\par
(iii)
Let ${\cal K}_{\Sigma}$ denote the simplicial complex on the vertex set $[r]$
defined by
%%(5.3)%%
\begin{equation}\label{KSigma}
{\cal K}_{\Sigma}=\big\{\{i_1,\cdots ,i_k\}\subset [r]\ \vert \ 
{\bf n}_{i_1},\cdots ,{\bf n}_{i_k}\mbox{ span a cone}\in \Sigma\big\},
\end{equation}
%%%%
%%%%%
and let $q_{\Sigma}$ denote the positive integer given by}
%%%(5.4)%%
\begin{equation}\label{equ: qSigma}
q_{\Sigma}=\max \{s\in\Z_{\geq 1}\ \vert \ 
\mbox{Any $s$ vectors }{\bf n}_{i_1},{\bf n}_{i_2},\cdots ,{\bf n}_{i_s}
\mbox{ span a cone in }\Sigma\}.
\end{equation}
%%%
%Note that $q_{\Sigma}$ is a positive integer such that $1\leq q_{\Sigma}<r$.
\end{dfn}
%%%%%%%%%%%%%%
%%%%%%%%%%%%%
%%%%%%%%%%%%
%%%(Lemma 5.2)%%%
\begin{lmm}\label{lmm: KSigma}
$U({\cal K}_{\Sigma})=\C^r\setminus \ZS $, and
$\rmin =q_{\Sigma}+1$.
\end{lmm}
%%%(Proof of Lemma 5.2)%%
\begin{proof}
From  (\ref{eq: ZS}), (\ref{zs}), (\ref{arrangement}) and (\ref{KSigma}),
it is clear that $U({\cal K}_{\Sigma})=\C^r\setminus \ZS $.
Recalling the definitions (\ref{dim rmin}) and (\ref{equ: qSigma}),
we easy see that $\rmin =q_{\Sigma}+1$.
\end{proof}
%%%%%%%%%%%%%%%%%%%%%%%

%%%%%
%%(Lemma 5.3)%%
\begin{lmm}[\cite{BP}]
\label{lmm: BP}
%%%%%%
If $K$ is a simplicial complex on the vertex set $[r]$,
 there is a homotopy equivalence 
${\cal Z}_K\simeq U(K)$ and the space ${\cal Z}_K$ is a homotopy fibre
of the inclusion map 
$DJ(K)\stackrel{\subset}{\rightarrow}(\CP^{\infty})^r.$
\end{lmm}
%%%%%
\begin{proof}
This follows from \cite[Corollary 6.30, Theorem 8.9]{BP}.
\end{proof}
%%
%%
%%(Lemma 5.4)%%%%%%%%%%%%%%%%%%%
\begin{lmm}\label{lmm: connectivity}
%%%%%%%%%%%%%%%%%%%%%%%%%%%%%%%%
The space $U({\cal K}_{\Sigma})=\C^r\setminus \ZS$ is $2(\rmin -1)$-connected.
%%%%%
\end{lmm}
%%%%%%%%%%%%%%%%%%%%%%%%%%%%%%%%%
%%%
%%%
%%(Proof of Lemma 5.4)%%%%%%%%%
\begin{proof}
%%%%%%%%%%%%%%
Because $q_{\Sigma}=\rmin -1$, it suffices to show that
$U({\cal K}_{\Sigma})$ is $2q_{\Sigma}$-connected.
Note that the $2s$-skeleton of $(\CP^{\infty})^r$ is
$\bigcup_{i_1+i_2+\cdots +i_r=s}
\CP^{i_1}\times \cdots \times \CP^{i_r}$.
Since $1\leq q_{\Sigma}<r$, 
by (\ref{KSigma}) and (\ref{equ: qSigma}) we see that
$\bigcup_{(i_1,\cdots ,i_r)\in {\cal I}}
\CP^{i_1}\times \cdots \times \CP^{i_r}
\subset DJ({\cal K}_{\Sigma})$,
where 
${\cal I}=\{(i_1,\cdots ,i_r)\ \vert \ i_j\in \{0,\infty\},\ 
\mbox{card}(\{i_j\ \vert \  1\leq j\leq r, i_j=\infty\})=q_{\Sigma}\}.$
Hence,
$DJ({\cal K}_{\Sigma})$ contains the $2q_{\Sigma}$-skeleton 
of $(\CP^{\infty})^r$.
Since $(\CP^{\infty})^r$ has no odd dimensional cells,
$DJ({\cal K}_{\Sigma})$ contains the $(2q_{\Sigma}+1)$-skeleton of $(\CP^{\infty})^r.$
Thus
the pair $((\CP^{\infty})^r,DJ({\cal K}_{\Sigma}))$ is
$(2q_{\Sigma}+1)$-connected.
Hence, by Lemma \ref{lmm: BP},  ${\cal Z}_{{\cal K}_{\Sigma}}$ is
$2q_{\Sigma}$-connected and  so is $U({\cal K}_{\Sigma})$.
\end{proof}
%%(End of proof of Lemma 5.4)%%%%
%%%%%(Remark 5.5)%%%
\begin{rmk}\label{rmk: complete}
%%%%%%%
{\rm The assertion of Lemma  \ref{lmm: connectivity}
holds even if $\Sigma$ in not complete.}
\end{rmk}
\section{The proof of the main result.}\label{section: main result}
%%%%%%%%%%%%%%%%%%%%%%%%%%%%%
%%
%%
%%
%%
%%
%%%%%%%%%%%%%%%%%%
In this section, we 
give the proofs of Theorem \ref{thm: AKY1-stable}
and the main result (Theorem \ref{thm: I}) by using Lemma \ref{lmm: connectivity}. 
 
%%(Definition 6.1)%%%%%%
\begin{dfn}
%%%%%
{\rm
%%%%%%%
Define a map
$\iota_D:A_D(m,\XS ;g)\to \Omega^m(\C^r\setminus \ZS)$ by
%%(6.1)%%
\begin{equation}\label{j'}
%%%%%%%
\iota_D(f_1,\cdots ,f_r)(\textbf{\textit{x}})=
(f_1(\textbf{\textit{x}}),\cdots ,f_r(\textbf{\textit{x}}))
\quad
\mbox{for }\textbf{\textit{x}}\in S^m.
\end{equation}
%%%%%%
}
\end{dfn}
%%%%%%%%%%%%%
\par\vspace{2mm}\par
%%%%%%%%%%%%
Consider the natural toric morphism
$p_{\Sigma}:\C^r\setminus \ZS \to \XS$.
By \cite[(8.6)]{BP} and \cite[Prop. 6.7]{Pa1}, we see that
there is an isomorphism $\GS\cong \T^{r-n}$
and that
the group $\GS$ acts on $U({\cal K}_{\Sigma})$
freely.
Hence, we have a  fibration sequence
%%%%%%%%%%%%%%%
%%%(6.2)%%%%%%%
\begin{equation}\label{fiber sequence}
%%%%%%%%%%%%%%%
\T^{r-n}\stackrel{}{\longrightarrow} U({\cal K}_{\Sigma})=\C^r\setminus \ZS 
\stackrel{p_{\Sigma}}{\longrightarrow} \XS .
\end{equation}
%%%
%%%%
Let $\gamma_m:S^m\to\RP^m$ be the double covering
and 
$
\gamma^{\#}_m:\Map^*(\RP^m,\XS)\to \Omega^m\XS
$ 
the map $\gamma^{\#}_m(f)=f\circ \gamma_m$.
%%%
%%%%%%%%%%%%%%%%%%%%%
Assume that $m\geq 2$, and consider the commutative diagram:
%%(6.3)%%%
\begin{equation}\label{diagram: key}
\begin{CD}
%%%%%%%
A_D(m,\XS;g) @>i_D^{\p}>> F(\RP^m,\XS ;g) @>i^{\p}>\subset>\Map^*(\RP^m,\XS)
\\
@V{\iota_D}VV @. \Vert @.
\\
\Omega^mU({\cal K}_{\Sigma}) @>\Omega^mp_{\Sigma}>\simeq> \Omega^m\XS
@<\gamma_m^{\#}<< \Map^*(\RP^m,\XS )
%%%%%%
\end{CD}
\end{equation}
%%%%%%%
%%%%%%%%%%
%\par
%%
Let $D_*(d_1,\cdots ,d_r;m)$ denote the positive integer defined by
%%%%%%
%%(6.4)%%
\begin{equation}\label{D*number}
D_*(d_1,\cdots ,d_r;m)=(2r_{\rm min} -m-1)\big(\lfloor \frac{\db +1}{2}\rfloor
+1\big) -1
\end{equation}
%%%%%%%%
where $\lfloor x\rfloor$ denotes the integer part of a real number $x$.
%%%
%%%
%%%
%%%%(Theorem 6.2)%%
\begin{thm}\label{thm: AKY1}
%%%%%%%%%%%%%%%%%%%
If $1\leq m\leq 2(\rmin -1)$, the map
$\iota_D:A_D(m,\XS;g) \to \Omega^mU({\cal K}_{\Sigma})$
is a homology equivalence through dimension
$D_*(d_1,\cdots ,d_r;m)$.
\end{thm}
%%%%%%%%%%%%%%%%%%%
%%%(Remark 6.3)%%%%
\begin{rmk}\label{rmk: 5.4}
%%%%%%%%%%%%%%%%
{\rm The assertion of Theorem \ref{thm: AKY1} holds
even if the condition (\ref{H(k)}.2) is not satisfied.}
%%%%%%%%%%%%%%%
\end{rmk}
%%%%%%%%%%%%%%%
%\par\vspace{2mm}\par
%%%
We postpone the proof of Theorem \ref{thm: AKY1}
to \S \ref{section: Vassiliev}, 
and
complete the proofs of 
Theorem \ref{thm: AKY1-stable} and Theorem \ref{thm: I}
by assuming it. 
%%%%
\par\vspace{2mm}\par
%%%%
%%%%
%%%(Proof of Theorem 4.6)%%%%%
\begin{proof}[Proof of Theorem \ref{thm: AKY1-stable}]
%%%%%%%%%%%%%
We set $D_*=D_*(d_1,\cdots ,d_r;m)$ and
let ${\Bbb F}$ denote  the field $\Z/p$ ($p$: prime) or $\Bbb Q$.
Since $m\geq 2$, by  (\ref{fiber sequence})
the map $\Omega^mp_{\Sigma}$ is a homotopy equivalence.
Then, from the diagram (\ref{diagram: key}) and Theorem \ref{thm: AKY1},
we see that the map $\gamma_m^{\#}\circ i^{\p}$
induces an epimorphism on homology groups
$H_k(\ ,{\Bbb F})$ for any $k\leq D_*$.
However, since there is a homotopy equivalence 
$F(\RP^m,\XS;g)\simeq \Omega^m\XS$,
$\dim_{\Bbb F}H_k(F(\RP^m,\XS ;g),{\Bbb F})=
\dim_{\Bbb F}H_k(\Omega^m\XS ,{\Bbb F})<\infty$ for any $k$.
%%%%
Therefore, the map $\gamma^{\#}_m\circ i^{\p}$ induces an isomorphism on
$H_k(\ ,\Bbb F)$ for any $k\leq D_*$.
Thus, from  the diagram (\ref{diagram: key}), we see that so does the map
$i_D^{\p}$.
It then follows from the universal coefficient theorem that 
the map
$i_D^{\p}$ is a homology equivalence through dimension 
$D_*=D_*(d_1,\cdots ,d_r;m).$
%%%
Because
$\lim_{k \to \infty}D_*(d_1+ka_1,\cdots ,d_r+ka_r;m)=\infty$,
the diagram
(\ref{diagram_sd}), implies that the map
$s_{D,\infty}$ is a homology equivalence.
%This completes the proof of Theorem \ref{thm: AKY1-stable}.
\end{proof}
%%%%%%%%(End of Proof of Theorem 4.6)%%%%%
%%%
%%%
%%%
%%%(PROOF of The Main Result)%%%%%%%%%%%%
%%%(Proof of Theorem 1.4)%%%%%%%%%%%%%%%
%%%%%%%%%
\begin{proof}[Proof of Theorem \ref{thm: I}]
%%%%%%%%%
%%%%%%%%%

The assertion easily follows from  (\ref{diagram_sd}), Theorem \ref{thm: AKY1-stable}
and Theorem \ref{thm: sd}.
%%%
%%%%
\end{proof}
%%%%

%%%%%
%%%%%(SECTION 7)%%%%%%%%%%%%%%%%%%%%%%%%%%%%%%%%%%%%%%%%
\section{The Vassiliev spectral sequence.}\label{section: Vassiliev}
%%%%%%%%%%%%%%%%%%%%%%%%%%%%%%%%%%%%%%%%%%%%%%%%%%%%%%%%
%%%
%%%
%In this section we prove Theorem \ref{thm: AKY1} by using the
%Vassiliev spectral sequence \cite{Va}.
%%%%%%%%%%%%%%
%\par\vspace{2mm}\par
%%%%%%%%%%%%%%
%%%%%%%%%%%%%%%%%%%%%%

\paragraph{Spectral sequences induced from the Veronese simplicial resolution.}
Let $Z_D$ be the tautological normalization of $\Sigma_D$, and let
$\pi_D:Z_D\to \Sigma_D$ denote the first projection as in (ii) of
Definition \ref{Def: 3.1}.
%%%
%\par
%%
Let $({\cal Z}^D,\pi_D:{\cal Z}^D\to \Sigma_D)$
denote the (degenerate) simplicial resolution
of the surjective map $\pi_D :Z_D\to \Sigma_D$ 
defined from the (generalized) Veronese embedding as in
\cite[page 782]{AKY1}.
We have the following natural filtration
$$
\phi ={\cal Z}^D_0\subset {\cal Z}^D_1=\Sigma_D\subset
{\cal Z}^D_2\subset {\cal Z}^D_3\subset 
\cdots \subset \bigcup_{k\geq 1}{\cal Z}^D_k={\cal Z}^D.
$$
%}
%\end{dfn}
%%%%%%%%
By the same method as in (\ref{Al}) and 
(\ref{SS}), we obtain a spectral sequence
%%%
%%%
%%%(7.1)%%
\begin{equation}\label{SSS1}
%%%
\{\hat{E}^t_{k,s},\ \hat{d}^t:\hat{E}^t_{k,s}\to \hat{E}^t_{k+t,s+t-1}\}
\quad
\Rightarrow 
\quad
H_{s-k}(A_D(m,\XS;g),\Z)
\end{equation}
%%%%
such that
%%%
%%%()%%
%\begin{equation}\label{6.2}
%%%%%%%%%%
$\hat{E}^1_{k,s}=\tilde{H}_c^{2N_D +k-s-1}({\cal Z}^D_k\setminus {\cal Z}^D_{k-1},\Z).$
%\end{equation}
%%%%%%%%%%%%
%%%
%%%
%%%(Lemma 7.1)%%%
\begin{lmm}\label{lmm: hatE1}
%%%%%%%%%%%%%%%%
\begin{enumerate}
%%%%
\item[$\I$]
If $1\leq k\leq \lfloor \frac{\db +1}{2}\rfloor$, there is a natural isomorphism
$$
\hat{E}^1_{k,s}\cong 
\tilde{H}_c^{2rk-s}(C_k,\pm \Z)
%H_{s-(2\rmin -m)k}(C_k,\pm \Z).
$$
\item[$\II$]
If $r<0$ or $s<0$ or $s\leq (2\rmin -m)k-1$, then
$\hat{E}^1_{k,s}=0$.
\end{enumerate}
\end{lmm}
%%%%%%%%%%%%%%%
\begin{proof}
%%%%%%%%%%%%%%%
(i)
By using the same argument as in the proof of Lemma \ref{Ck}
and \cite[Lemma 4.4 and Lemma 4.6]{AKY1},
we can show that
${\cal Z}^D_k\setminus {\cal Z}^D_{k-1}$ is an open disk bundle over
$C_k$ with rank $l_{D,k}$ if $1\leq k\leq \lfloor \frac{\db +1}{2}\rfloor$.
(Note that the projection $\pi_k :{\cal Z}^D_k\setminus {\cal Z}^D_{k-1}
\to C_k$ is well-defined only if $k\leq \lfloor \frac{\db +1}{2}\rfloor$,
because the condition (i) of (\ref{def: 2.1}$)_k$ is not satisfied
for $k>\lfloor \frac{\db +1}{2}\rfloor$.)
We can now prove the assertion (i) in exactly the same way as Lemma \ref{lemma: vector bundle*}.
%%%
\par
(ii)
It suffices to show that $\hat{E}^1_{k,s}=0$ if
$s\leq (2\rmin -m)k-1$ when $k,s\geq 0$.
%%
%Since $\dim \ZS =2(r-\rmin)$,
In general,
\begin{eqnarray*}
\dim ({\cal Z}^D_k\setminus {\cal Z}^D_{k-1})
&\leq&
(2N_D-2rk)+\dim C_k +(k-1)
\\
&=& (2N_D-2rk)+(\dim \ZS +m)k+(k-1)
\\
&=&
2N_D-(2\rmin -m)k+k-1
\end{eqnarray*}
%%%%
Since 
$2N_D+k-1-s>2N_D-(2\rmin -m)k+k-1$
$\Leftrightarrow$
$s\leq (2\rmin-m)k-1$
and
$\hat{E}^1_{k,s}=\tilde{H}_c^{2N_D +k-s-1}({\cal Z}^D_k\setminus {\cal Z}^D_{k-1},\Z)$,
$\hat{E}^1_{k,s}=0$ if $s\leq (2\rmin-m)k-1$ and (ii) follows.
%%%
%%%%%%%%%%%%%
\end{proof}
%%%%%%%%%%%%%

%%%%
\paragraph{The Vassiliev spectral sequence.}
%%%%%%%%
We now 
recall the spectral sequence constructed by V. Vassiliev
\cite[page 109--115]{Va}.
\par
%%%%
From now on, we will assume that $m\leq N$ and that $X$ is 
a finite dimensional $N$-connected simplicial complex $C^{\infty}$-imbedded in $\R^L$.
We regard $S^m$ and $X$ as subspaces $S^n\subset \R^{m+1},\   X\subset \R^L$,
respectively.
We also choose and fix a map
$\varphi :S^m\to X$.
%%
%%%%%
Observe that $\Map (S^m,\R^{L})$ is a linear space
and consider
the complements
%$$
%\begin{cases}
${\frak A}_m(X)=\Map (S^m,\R^{L})\setminus \Map (S^m,X)$ and
%\\
$\tilde{\frak A}_m(X)=\Map^* (S^m,\R^{L})\setminus \Map^*(S^m,X).$
%\end{cases}
%$$
%%
Note that  ${\frak A}_m(X)$ consists of all continuous maps
$f:S^m\to \R^{L}$ intersecting $\R^L\setminus X.$
%%%%
We will denote by $\Theta^d_{\varphi}(X)\subset \Map (S^m,\R^{L})$ the subspace
consisting of all maps $f$ of the forms $f=\varphi +p$,
where $p$ is the restriction to $X$ of a polynomial map
$S^m\to \R^L$ of total degree $\leq d$.
%%%
%%
Let
$\Theta^d_X\subset \Theta^d_{\varphi}(X)$ denote the subspace
consisting of all $f\in \Theta_{\varphi}^d(X)$
intersecting $\R^L\setminus X.$ 
%which
% intersects ${\bf 0}_{n+1}$.
In \cite[page 111-112]{Va} Vassiliev uses the space 
$\Theta^d(X)$ as a finite dimensional approximation of 
${\frak A}_m(X)$.\footnote{%
%%%%%%%%%%%%%%%%%%%%%%%%%%%%%
%%%%%%%(FootNote 10)%%%%%%%%%
%%%%%%%%%%%%%%%%%%%%%%%%%%%%%
Note that the proof 
 of this fact given by Vassiliev makes use of the Stone-Weierstrass theorem, 
 so, although we are now not using the stable result of \cite[Theorem 2.1]{AKY1}, something like it is also implicitly  involved here. }
%%%%%%%%%%%%%%
%%%%%%%%%%%%%%
\par
%%%
Let $\tilde{\Theta}^d_X$ denote the subspace of $\Theta^d_X$ consisting of all maps 
$f\in \Theta^d_X$ which preserve
the base points.
%%%
By a variation of the preceding argument, Vassiliev also shows that  
$\tilde{\Theta}^d_X$ can be used as a finite dimensional approximation of  
$\tilde{\frak A}_m(X)$ \cite[page 112]{Va}.
\par
%%%%%%%%
%%
%By choosing a sufficiently fine  polynomial approximation of $\tilde{\frak A}_m^n$, 
Let ${\cal X}_d\subset \tilde{\Theta}^d_X\times \R^{L}$ denote the %tautological normalization of $\tilde{\Theta}^k$
subspace consisting of all pairs 
$(f,\alpha)\in\tilde{\Theta}^d_X \times \R^{L}$
such that $f(\alpha )\in \R^L\setminus X$, and
let $p_d:{\cal X}_d\to \tilde{\Theta}^d_X$ be the projection onto the first factor.
Then,  by making use of (non-degenerate) simplicial resolutions of the surjective maps 
$\{p_d:d\geq 1\}$,
one can construct a simplicial resolution 
$\{\tilde{\frak A}_m(X)\}$ of $\tilde{\frak A}_m,$ 
whose cohomology is naturally isomorphic to the homology of $\Map^*(S^m,X)=\Omega^mX$. 
%%%%
From the natural filtration 
%%%()%%
%\begin{equation*}\label{geometric resolution}
%%%%%%%
$\dis F_1\subset F_2\subset F_3\subset 
\cdots \subset \bigcup_{d=1}^{\infty}F_d=\{\tilde{\frak A}_m(X)\},$
%\end{equation*}
%%%
we obtain the associated spectral sequence
(for $m\leq N$)
%%
%%
%%%(7.2)%%%
\begin{equation}\label{SSSSS}
%%%%%%%
\{E^t_{k,s},d^t:
E^t_{k,s}\to
E^t_{k+t,s+t-1}\}
\Rightarrow
H_{s-k}(\Omega^mX,\Z).
\end{equation}
%%%
\par\vspace{2mm}\par
%%%%%%%
%%%%%%%
%%%%%%%
Since $U({\cal K}_{\Sigma})$ is $2(\rmin -1)$-connected (by Lemma \ref{lmm: connectivity}),
we may assume that $m\leq N=2(\rmin -1)$ and consider the spectral sequence (\ref{SSSSS}) for
$X=U({\cal K}_{\Sigma})=\C^r\setminus \ZS$.
%\par
%%%%
Note that the construction of this simplicial resolution is almost identical to that of the resolution ${\cal Z}^D$.
The only difference between the two concerns the following two points.
%%%%%%%
%%%%(Point 1)%%%%
First, we use all polynomials passing through $\ZS$
of total degree $\leq d$
 instead of homogenous polynomials passing through $\ZS$
 satisfying the condition
 (\ref{H(k)}.2),
 for $d$ some fixed large integer.
%%%% 
%%%%
%%%%(Point 2)%%% 
%%%%%%%
Secondly we use a family of embeddings satisfying the condition
(\ref{def: 2.1}$)_k$ instead of a fixed embedding.
%%%%%
%\par
However, since $\hat{E}^1_{k,s}$ is determined independently of
the choice of embeddings if $d$ is sufficiently large, one can prove the following result 
by using  the same method
as in the case of the Veronese resolution. 
%by choosing the 
%sufficiently large degree $d$.
%%%%%%

%%%(Lemma 7.2: Vassiliev's result)%%%
\begin{lmm}[\cite{Va}]
\label{lemma: Vass}
%%%%%%%%%%%%%%%%
If $1\leq m\leq 2(\rmin -1)$,
there is a spectral sequence
%%%
%%%(7.3)%%
\begin{equation}\label{VSS}
%%%
\{E^t_{k,s},d^t:
E^t_{k,s}\to
E^t_{k+t,s+t-1}\}
\Rightarrow
H_{s-k}(\Omega^mU({\cal K}_{\Sigma}),\Z)
\end{equation}
%%%%%
satisfying the following two conditions:
%%%%%%
\begin{enumerate}
%%(i)%
\item[$\I$]
If $k\geq 1$,
$E^1_{k,s}=
\tilde{H}_c^{2rk-s}(C_k,\pm \Z).$
%%(ii)
\item[$\II$]
If $k<0$ or $s<0$ or $s\leq (2\rmin -m)k-1$, then $E^1_{k,s}=0$.
\qed
\end{enumerate}
\end{lmm}
%%
%%%(End of Lemma 7.2)%%%%%%%%%%%%%%%%%%%
\par\vspace{2mm}\par
%%%%%%
Now we can give the proof of Theorem \ref{thm: AKY1}.

%%%%%%%%%%%%%%%%%%%%%%%%%
%%(Proof of Theorem 6.2)%%%
\begin{proof}[Proof of Theorem \ref{thm: AKY1}]
%%%%%%%%%%%%%%%%%%%%%%%%%%%%%%%%%%%%%%%%%%%%%%%
%%%%%%%%%%%%%%%%%%%%%%%%%%%%
%%
Consider the  spectral sequences (\ref{SSS1}) and (\ref{VSS}).
Note that the image of the  map $\iota_D$ lies 
in a space of mappings that arise from restrictions of polynomial mappings
$\R^{m+1}\to \C^r=\R^{2r}$.
Since  ${\cal X}^D$ is a non-degenerate simplicial resolution,  the map
 $\iota_D$ naturally extends to a filtration
preserving map 
$\tilde{\pi}: {\cal X}^D\to \{\tilde{\frak A}_m(U({\cal K}_{\Sigma}))\}$
between resolutions.
%%%
By \cite[Lemma 2.2]{KY4} there is a filtration
preserving homotopy equivalence 
$q^{\Delta} :{\cal X}^D\stackrel{\simeq}{\rightarrow}{\cal Z}^D$.
The filtration preserving maps
$
\begin{CD}
{\cal Z}^D @<q^{\Delta}<\simeq< {\cal X}^D @>\tilde{\pi}>>
\{\tilde{\frak A}_m(U({\cal K}_{\Sigma}))\}
\end{CD}$
%%%%%%%%
induce a homomorphism of spectral sequences
$\{ \theta^t_{k,s}:\hat{E}^t_{k,s}\to E^t_{k,s}\},$
where
%%%
$$
\{\hat{E}^t_{k,s},\hat{d}^t\}\Rightarrow H_{s-k}(A_D(m,\XS;g),\Z)
\quad
\mbox{ and }\quad
\{E^t_{k,s},d^t\}\Rightarrow H_{s-k}(\Omega^m U({\cal K}_{\Sigma} ),\Z).
$$
%%%%
Then by the naturality of the Thom isomorphism and
the  argument used in the proof of Lemma \ref{lemma: Thom},
we can show that
$\theta^1_{k,s}:\hat{E}^1_{k,s}\stackrel{\cong}{\rightarrow}
E^1_{k,s}$ is an isomorphism for any $s$ 
as long as $k\leq \lfloor \frac{\db+1}{2}\rfloor .$
%%%
%if $1\leq k\leq \lfloor \frac{\db+1}{2}\rfloor$, 
%we have the following commutative diagram
%%%%%%%%
%$$
%%%%%%%%
%\begin{CD}
%\hat{E}^1_{k,s} 
%@>T>\cong> H_{s-(2\rmin -m)k}(C_k,\pm \Z)
%\\
%@V{\theta}^1_{k,s}VV  \Vert @.
%\\
%{E^1_{k,s}} 
%@>T>\cong> 
%H_{s-(2\rmin -m)k}(C_k,\pm \Z)
%\end{CD}
%$$
%Hence, if $k\leq \lfloor \frac{\db+1}{2}\rfloor$,
%$\theta^1_{k,s}:\hat{E}^1_{k,s}\stackrel{\cong}{\rightarrow}
%E^1_{k,s}$ 
%is an isomorphism for any $s$.
By Lemma \ref{lmm: hatE1} and
\ref{lemma: Vass}, we see that
$\theta^{\infty}_{k,s}:
\hat{E}^{\infty}_{k,s}\stackrel{\cong}{\rightarrow}
E^{\infty}_{k,s}$ is always an isomorphism for any $s$
if $k\leq \lfloor \frac{\db+1}{2}\rfloor$.
%%%%%%%%%%%%%%%%%%%%%
%%%%%%%%
%%%%%%%
%%%%%%%%%%
%%%
Now, consider the positive integer $D_{\rm min}$:
$$
D_{\rm min}=\min\{
N\in \Z_{\geq 1}\ \vert \ N\geq s-k,\ 
s\geq (2\rmin -m)k,\ 
1\leq k<\lfloor \frac{\db +1}{2}\rfloor +1
\}.
$$
%%%%
Clearly  $D_{\rm min}$ is the
largest integer $N$ which satisfies the inequality
%%%%%%
$(2\rmin -m)k> k+N$
for $k=\lfloor\frac{\db +1}{2}\rfloor +1$.
Hence,
$D_{\rm min}
=
(2\rmin -m-1)(\lfloor\frac{\db +1}{2}\rfloor +1)-1=D_*(d_1,\cdots ,d_r;m ).$
%\end{equation}
%%%%%%%%%%%
%%%%%%%%%%%
We note that, for dimensional reasons,
$\theta^{\infty}_{k,s}:
\hat{E}^{\infty}_{k,s}
\stackrel{\cong}{\rightarrow} 
E^{\infty}_{k,s}$ is always
an isomorphism when
$k\leq \lfloor \frac{\db +1}{2}\rfloor$ and $s-k\leq D_*(d_1,\cdots ,d_r;m )$.
%%
%\par
Note also that by Lemma \ref{lmm: hatE1} and
Lemma \ref{lemma: Vass},
$\hat{E}^1_{k,s}=E^1_{k,s}=0$ when
$s-k\leq D_*(d_1,\cdots ,d_r;m )$ and $k>\lfloor\frac{\db +1}{2}\rfloor$.
Hence, we see that
$\theta^{\infty}_{k,s}:
\hat{E}^{\infty}_{k,s}\stackrel{\cong}{\rightarrow} 
E^{\infty}_{k,s}$ is always
an isomorphism if $s\leq k+D_*(d_1,\cdots ,d_r;m)$.
Thus, it follows from the Comparison Theorem of spectral sequences that
the map $\iota_D$ is a homology equivalence through dimension
$D_*(d_1,\cdots ,d_r;m)$.
%%%%
\end{proof}
%%%%(End of Proof of Theorem 6.2)%%
%%%

%%%%(Appendix SECTION 8)%%%%%%%
\section{
Some facts and  conjectures.}\label{section: Appendix}
%%%%%%%%%%%%%%%%%%%%

%%%%%
%\subsection{Basic facts.}
%%%%%%
\paragraph{A basic lemma.}
The aim of the first part of this  section is to provide a simple basic lemma used in this paper for which we do not know any convenient reference.
%%%
\par\vspace{1mm}\par

%%%%(Polytopes)%%%%%%%
%%%%%%%%%%%%%
%\paragraph{Polytopes. }
%%%%%%%%%%%%%%%%%%%%%%%
%First, we summarize the basic fact about polytopes
%(\cite{BP}, \cite{Z}).
%\par
%%%%%%%(i)%%%
%%%%%%%%
%%%%%
\par
%%%%(Lemma 8.1: A1)%%%%
%%%%%%%%%%%%%%%%%%
\begin{lmm}\label{lmm: A1}
%%%%%%%%%%%%%%%%%%
Let $K$ be a CW complex and
$X=K\cup_fe^m$ with $\dim K<m$.
For $g\in \Map^* (K,Y)$, let
$F(X,Y;g)$ denote the space
given by
$F(X,Y;g)=\{h\in \Map^* (X,Y): h\vert K=g\}$.
If $F(X,Y;g)\not= \emptyset$,
there is a homotopy equivalence
$F(X,Y;g)\simeq  \Omega^mY$.
\end{lmm}
%%%%%%
\begin{proof}
%%%%
By using the characteristic map of the top cell in $X$, 
$F(X,Y;g)$ can be identified with the space of all based maps
$h:D^m \to Y$ which restrict to the same fixed map on the boundary $S^{m-1}$,
and it can be regarded as the fiber of the map
$r:\Map^* (D^m,Y)\to \Map^*(S^{m-1},Y)=\Omega^{m-1}Y$ given by
$r(h)=h\vert S^{m-1}$.
Since $r$ is a fibration with fiber $\Omega^mY$
%(\cite[page 97, Theorem 2]{Sp})
and $\Map^*(D^m,Y)$ is contractible,
there is a homotopy equivalence
$F(X,Y;g)\simeq \Omega^mY.$
\end{proof}
%%%%%%%%

%%%%%(Lemma 8.2: A2)%%
%\begin{lmm}\label{lmm: A2}
%%%%%%%%%%%%%%%%%
%If $\sum_{k=1}^rd_k{\bf n}_k={\bf 0}$ and
%$(f_1,\cdots ,f_r)\in A_{D,\Sigma}(m)$, the map 
%$j_D^{\p}$
%is well defined. 
%\end{lmm}
%%%%%%%%%%%%%
%\begin{proof}
%%%%%%%%%%%%%
%Let $\lambda\in \R^*$.
%Then because
%$\prod_{k=1}^r (\lambda^{d_k})^{\langle {\bf m},{\bf n}_k\rangle}=
%\lambda^{\langle {\bf m},\sum_{k=1}^rd_k{\bf n}_k
%\rangle}=1$
%for any ${\bf m}\in \R^n$,
%$(\lambda^{d_1},\cdots ,\lambda^{d_r})\in \GS$.
%%%%%
%Since
%$(f_1(\lambda \textbf{\textit{x}}),\cdots ,f_r(\lambda \textbf{\textit{x}}))=
%(\lambda^{d_1}f_1(\textbf{\textit{x}}),\cdots ,\lambda^{d_r}f_r(\textbf{\textit{x}})),$
%$
%[f_1(\lambda \textbf{\textit{x}}),\cdots ,f_r(\lambda \textbf{\textit{x}})]=
%[f_1(\textbf{\textit{x}}),\cdots ,f_r(\textbf{\textit{x}})]$
%in
%$\XS$
%for any $(\lambda ,\textbf{\textit{x}})\in \R^*\times (\R^{m+1}\setminus
%\{{\bf 0}\})$,
%hence $j_D$  is well-defined.
%%%%%%%%%%%%
%\end{proof}
%%%%%%%%%%%
%%%%(Corollary 7.3)%%
%\begin{crl}\label{cor: A3}
%%%%%%%%%%%%%%%%%
%If $\sum_{k=1}^ra_k{\bf n}_k={\bf 0}$ and
%$(f_1,\cdots ,f_r)\in A_{D,\Sigma}(m)$ with $a_k\in \Z_{\geq 1}$,
%$$
%[\lambda^{a_1}f_1(\textbf{\textit{x}}),\cdots ,\lambda^{a_r}f_r(\textbf{\textit{x}})]
%=[f_1(\textbf{\textit{x}}),\cdots ,f_r(\textbf{\textit{x}})]
%\quad
%\mbox{in }\XS
%\
%\mbox{
%for $(\lambda ,\textbf{\textit{x}})\in \C^*\times (\R^{m+1}\setminus
%\{{\bf 0}\}).$}
%\qed
%$$
%\end{crl}
%%%%%%%%%%
%%
%%(Minimal degree)%%%%%%%%%%%
\paragraph{Minimal degree of algebraic maps.}
%%%%%%%%%%%%%%%%%%%%%%%%%%%%%
In the second part of this section we define  the minimal degree of an algebraic map $\RP^m$ to $\XS$. It plays no role in the current paper but we think it is of sufficient independent interest and hope to make use of it in the future. 
First, we give the proof of Proposition \ref{prp: the minimal degree}
stated in \S \ref{section 1}.
%\par%\vspace{1mm}\par
%Remark that
%one of the differences between the complex case (considered in \cite{MV}) and the case considered here, lies in the concept of \lq\lq degree\rq\rq\ of an algebraic map. 
%
% In the complex case the degree of an algebraic mapping can be defined as the degree of its representation in Cox coordinates. 
% 
% In our case  \lq\lq the degree\rq\rq\ of an algebraic map cannot be defined unambiguously since two $r$-tuples of polynomials of different degrees can represent the same map. 
%However, we can show that each algebraic map $f:\RP^m\to \XS$ has a well defined  \lq\lq minimal degree\rq\rq  as follows.
%More precisely, the minimal degree of an algebraic map $f$
%is uniquely determined by itself and not on its representatives
%(cf. Proposition \ref{prp: the minimal degree}).
%Although the concept of the minimal degree is very important,
%the topology of the space of algebraic maps with a fixed minimal degree
%is very complicated to analyze.
%So
%we shall not make here  any use of this concept (besides defining it
%in  Definition \ref{dfn: degree}).

%\par\vspace{1mm}\par

%%%(Proof of Prop. 1.2)%%%%%
\begin{proof}[Proof of Proposition \ref{prp: the minimal degree}]
%%%%%%%%%%%%%%%%%%%%%%%%%%
%%%%%%%%%%%%%%%%%%%%%%%%%%
Since $f:\RP^m\to \XS$ is an algebraic map, 
there is a Zariski open subset $U$ of $\CP^m$ such that $U$ contains 
the set of $\mathbb{R}$-valued points of $\RP^m$, which we also denote by $\RP^m$ by abuse of notation, 
and there is a regular map $\varphi:U\to \XS$ such that $f=\varphi|_{\mathbb{R}P^m}$.
Without loss of generality
%By replacing $U$ if necessary, 
we may assume that $U$ is the largest open subset of $\CP^m$ where $f$ is defined.
Then $\CP^m\setminus U$ has codimension at least two in 
$\CP^m$,
since $X_{\Sigma}$ is proper and $\CP^m$ is normal.
The following proof is almost the same as that of \cite[Theorem 3.1]{C},
but in our case, the map $\varphi$ is defined only on the Zariski open subset $U$ of $\CP^m$,
which requires some modifications.
Let $\pi:\mathbb{C}^{m+1}\setminus \{{\bf 0}\}\to \CP^m$ be the Hopf fibering,
and set $\tilde{U}=\pi^{-1}(U)$.
We also denote by $\pi$ its restriction $\tilde{U}\to U$, and
let
$p_{\Sigma}:\mathbb{C}^r\setminus Z_{\Sigma}\to \XS$ be the natural toric morphism as in (\ref{fiber sequence}).
%%%%%
%From now on, remark that $\varphi$ is defined only on a dense open subset 
%$U$ of $\CP^m$, and
%%
We shall show the existence of $D$ and $(f_1,\cdots,f_r)$.
%construct a map
%$F:\tilde{U}\to \mathbb{C}^r\setminus Z_{\Sigma}$ given by $(f_1,\dots,f_r)$
%such that $\varphi\circ\pi=p_{\Sigma}\circ F$.
\par
Let 
$(\mathcal{O}_{X_{\Sigma}}(D_{\rho}),\iota_{\rho},c_{\chi^{\bf m}})
_{\rho\in \Sigma (1),{\bf m}\in \Z^n}
$ denote
the universal $\Sigma$-collection defined in \cite[page 252]{C}
and let 
$(L_{\rho}, u_{\rho}, c_{\bf m})
_{\rho\in \Sigma (1),{\bf m}\in \Z^n}
$ 
be its pull back by $\varphi$.
%%
%%%
Since $\varphi$ is a regular map, $L_{\rho}$ is an algebraic line bundle on $U$.
Moreover, since $\CP^m\setminus U$ has codimension at least two in $\CP^m$,
 $L_{\rho}$ can be extended to a line bundle on $\CP^m$
which is isomorphic to $\mathcal{O}_{\CP^m}(d_{\rho})$
for some integer $d_{\rho}$.
%%%
This isomorphism induces an isomorphism 
$H^0(U,L_{\rho})\cong H^0(\CP^m,\mathcal{O}_{\CP^m}(d_{\rho}))$.
Let $f'_{\rho}\in H^0(\CP^m,\mathcal{O}_{\CP^m}(d_{\rho}))$ 
denote the element corresponding to $u_\rho\in H^0(U,L_{\rho})$ by the above isomorphism.
The isomorphisms $c_{\bf m}$ on $U$ can be extended to those on $\CP^m$
which induce isomorphisms 
$c'_{\bf m}:
\otimes_{\rho}\mathcal{O}_{\CP^m}(d_{\rho})^{\langle {\bf m},\mathbf{n}_{\rho}\rangle}
\cong 
\mathcal{O}_{\mathbb{C}\P^m}$
on $\mathbb{C}\P^m$.
%%%%%
In terms of (\ref{H(k)}.1), this implies that $\sum_{k=1}^rd_k{\bf n}_k={\bf 0}$.
A collection $(\mathcal{O}_{\CP^m}(d_{\rho}), f'_{\rho}, c'_{\bf m})$ 
is not necessarily a $\Sigma$-collection on $\CP^m$
because sections $f'_{\rho}$ do not satisfy the non-degeneracy 
condition outside $U$.
%%%
However its restriction $(\mathcal{O}_{\CP^m}(d_{\rho}), f'_{\rho}, c'_{\bf m})|_U$ to $U$ is 
a $\Sigma$-collection on $U$. 
%%%%
For each ${\bf m}\in \Z^n$, we have 
a canonical isomorphism $c_{\bf m}^{\textrm{can}}:\otimes_{\rho}\mathcal{O}_{\CP^m}(d_{\rho})^{\langle {\bf m},\mathbf{n}_{\rho}\rangle}
\stackrel{\cong}{\rightarrow} \mathcal{O}_{\mathbb{C}\P^m}$
as in the proof of  \cite[Theorem 3.1]{C}.
Then, as in the final paragraph of the proof of \cite[Theorem 3.1]{C},
by taking $\lambda_p\in \mathbb{C}$ suitably and setting $f_{\rho}=\lambda_{\rho}f'_{\rho}$,
we have an equivalence $(\mathcal{O}_{\CP^m}(d_{\rho}), f'_{\rho}, c'_{\bf m})|_U\sim (\mathcal{O}_{\CP^m}(d_{\rho}), f_{\rho}, c^{\textrm{can}}_{\bf m})|_U$
of $\Sigma$-collection on $U$.
Now  define a morphism $F:\tilde{U}\to \mathbb{C}^r\setminus Z_{\Sigma}$
by $F(\textbf{\textit{x}})=(f_{\rho}(\textbf{\textit{x}}))_{\rho\in \Sigma(1)}$ for $\textbf{\textit{x}}\in \tilde{U}$.
Note that the non-degeneracy of $f_{\rho}$ on $U$ ensures that 
$(f_{\rho}(\textbf{\textit{x}}))_{\rho\in \Sigma(1)}\notin Z_{\Sigma}$
for all $\textbf{\textit{x}}\in \tilde{U}$.
Since the pull back of $(\mathcal{O}_{\CP^m}(d_{\rho}), f_{\rho}, c^{\textrm{can}}_{\bf m})|_U$ to $\tilde{U}$ is 
$(\mathcal{O}_{\tilde{U}}, f_{\rho}, 1)$,
we see that $\varphi\circ\pi=p_{\Sigma}\circ F$.
By reindexing $({f_{\rho}})_{\rho\in \Sigma(1)}$ as $(f_1,\cdots,f_r)$ 
in accordance with (\ref{H(k)}.1) and setting $D=(d_1,\cdots,d_r)$,
we obtain an element $(f_1,\cdots ,f_r)\in A_{D,\Sigma}(m)$.
Since $\CP^m\setminus U$ has codimension at least two in $\CP^m$,
we see that 
$\dim_{\mathbb{C}}pr_1^{-1}(F)<m$, 
where $F=(f_1,\cdots ,f_r)$ and $pr_1$ is the projection as in \S 1. 
Therefore, we have $(f_1,\cdots ,f_r)\in A_{D,\Sigma}(m)^{\circ}$.
\par
Suppose that $(h_1,\cdots ,h_r)\in A_{D',\Sigma}(m)^{\circ}$ also represents the same algebraic map $f$
for some $r$-tuple $D'=(d'_1,\dots,d'_r)\in \Z^r$
such that $\sum_{k=1}^rd'_k{\bf n}_k={\bf 0}$.
Set $H=(h_1,\cdots ,h_r)$ and $V=\CP^m\setminus \pi(pr_1^{-1}(H))$.
Then $V\subset U$ and 
$(\mathcal{O}_{\CP^m}(d_{\rho}), f_{\rho}, c^{\textrm{can}}_{\bf m})|_V\sim (\mathcal{O}_{\CP^m}(d'_{\rho}), h_{\rho}, c^{\textrm{can}}_{\bf m})|_V$.
Since $\CP^m\setminus V$ has codimension at least two in $\CP^m$, 
the isomorphism $\mathcal{O}_{\CP^m}(d_{\rho})|_V\cong \mathcal{O}_{\CP^m}(d'_{\rho})|_V$ is given by a nonzero constant $\mu_{\rho}\in \mathbb{C}$.
Therefore $d_{\rho}=d'_{\rho}$ for all $\rho\in \Sigma(1)$ and we have $D=D'$.
Moreover, as in the proof of \cite[Theorem 3.1]{C}, this implies that 
$(\mu_1,\cdots ,\mu_r)\in\GS$ and $(h_1,\cdots ,h_r)=(\mu_1f_1,\cdots ,\mu_rf_r)$.
\end{proof}
%%(End of proof of prop.1.2)%%%
%%%%
By Proposition \ref{prp: the minimal degree}, we can define
the minimal degree of an algebraic map as follows.

%%%%(Definition of DEGREE)%%%%
%%%%%%%%%(Definition 8.3)%%%
\begin{dfn}\label{dfn: degree}
%%%%%%%%%%%%%%%%%%%%%%%%%%%%
{\rm
Let
$D=(d_1,\cdots ,d_r)\in \Z^r$ and let
$f:\RP^m\to \XS$ be an algebraic map. 
%\par
Then the map
$f$ is called
an algebraic map {\it of minimal degree} $D$ if
it can be represented as 
%a rational map (or regular map) of  
%the form
$f=j_D^{\p}(f_1,\cdots ,f_r)=[f_1,\cdots ,f_r]$ for some
$(f_1,\cdots ,f_r)\in A_{D,\Sigma}(m)^{\circ}.$
}
\end{dfn}
%%%(End of Definition 8.3)%%%%%%%%%
Note that the minimal degree 
depends only on the map itself
and not on its representative.

%%%%%%%%%%%%%%%%%%
%%%(Remark 8.4)%%
%%%%%%%%%%%%%%%%%%
\begin{rmk}\label{rmk: homeo}
%%%%
{\rm 
(i)
We denote by $\Alg_{D,{\rm min}}(\RP^m,\XS)\subset \Alg_D(\RP^m,\XS)$ the subspace
consisting of all
algebraic maps $f:\RP^m\to \XS$ of minimal degree $D$.
Since $A_{D,\Sigma}(m)^{\circ}$ is a $\GS$-invariant subspace of
$A_{D,\Sigma}(m)$,
let $\widetilde{A_{D}}(m,\XS)^{\circ}$ denote the orbit space
$\widetilde{A_{D}}(m,\XS)^{\circ}=A_{D,\Sigma}(m)^{\circ}/\GS.$
%(cf. (\ref{ADS})).
%%
Then
we can easily see that
there is a homeomorphism
%%%(8.1)%%
\begin{equation}
\widetilde{A_{D}}(m,\XS)^{\circ}
\stackrel{\cong}{\rightarrow}
\Alg_{D,{\rm min}}(\RP^m,\XS).
\end{equation}
%%%
\par
(ii)
Let
$\textbf{\textit{a}}=(a_1,\cdots ,a_r)\in (\Z_{\geq 1})^r$ and suppose that 
$\sum_{k=1}^ra_k{\bf n}_k={\bf 0}$.
Then if we set $\tilde{g}=\sum_{k=1}^rz_k^2$ and
$f=[f_1,\cdots ,f_r]\in \Alg_D(\RP^m,\XS)$,
then we can easily see that
$f=[\tilde{g}^{a_1}f_1,\cdots ,\tilde{g}^{a_r}f_r]\in 
\Alg_{D+\textbf{\textit{a}}}(\RP^m,\XS).$
Hence,
the space $\Alg_D(\RP^m,\XS)$ can be identified with the subspace of $\Alg_{D+\textbf{\textit{a}}}(\RP^m,\XS)$
by
%%%()%%
\begin{equation*}\label{equ: representation}
%%%%
\Alg_D(\RP^m,\XS) \stackrel{\subset}{\rightarrow} \Alg_{D+\textbf{\textit{a}}}(\RP^m,\XS);\qquad
[f_1,\cdots ,f_r] \mapsto [f_1\tilde{g}^{a_1},\cdots ,f_r\tilde{g}^{a_r}].
\end{equation*}
%%%
Note that $\Map_D(\RP^m,\XS)=\Map_{D^{\p}}(\RP^m,\XS)$ may happen
even if $D\not=D^{\p}$.
However, $\Alg_{D,{\rm min}}(\RP^m,\XS)\cap\Alg_{D^{\p},{\rm min}}(\RP^m,\XS)=
\emptyset$ if $D\not=D^{\p}$.
\par
(iii)
%It follows from the above proof of Proposition \ref{prp: the minimal degree} that
%we can see that
It may happen that $d_k<0$ for some $k$;
the section $f_k\in H^0(\CP^m,{\cal O}(d_k))$
%because $f_k$ is a global section of ${\cal O}_{\CP^m}(d_k)$,
 will be zero if $d_k<0$ 
(see Example \ref{exa: example}).}
\qed
%%%%%%
\end{rmk}
%%%%%

%%%%(Example 8.5)%%%%%%%%
\begin{exa}\label{exa: example}
%%%%%%%%%%%%%%%%%%%%%%%%%
{\rm
Let $H(k)$ and $\{\mathbf{n}_k\}_{k=1}^4$ be the Hirzebruch surface 
and its primitive generators as in Example~\ref{exa: Hirzebruch}.
Suppose that $(d_1,d_2,d_3,d_4)=(1,-k,1,0)$.
Then we have $\sum_{j=1}^4d_j\mathbf{n}_j=\mathbf{n}_1-k\mathbf{n}_2+\mathbf{n}_3=\mathbf{0}$.
Let $k\geq 0$ and we may regard 
$H(k)$ as the orbit space
%%%()%%
%\begin{equation}\label{equ: H(k)orbit}
%%%%%%%
$$
H(k)=\{(y_1,y_2,y_3,y_4)\in\C^4\ \vert \  (y_1,y_3)\not= (0,0),\ (y_2,y_4)\not= (0,0)\}/\GS ,
$$
%\end{equation}
%%%%%%
where
$\GS =\{(\mu_1,\mu_2,\mu_1,\mu_1^k\mu_2)\ \vert \ \mu_1,\mu_2\in \C^*\}$
as in (\ref{H(k)}).
Now consider the algebraic map $f:\RP^1\to H(k)$ 
defined by $f([x_0:x_1])=[(x_0^2+x_1^2)x_0,0,(x_0^2+x_1^2)x_1,1]$.
%where $[x_0:x_1]$ is the homogeneous coordinate of $\RP^1$.
%%
%%%
Then clearly $f\in \Alg_{(3,-3k,3,0)} (\RP^1,H(k))$.
However, by the $\GS$-action, we have  $f([x_0:x_1])=[x_0,0,x_1,1]$.
Hence $f\in \Alg_{(1,-k,1,0), \textrm{min}} (\RP^1,H(k))$.
%In particular, $d_2<0$ if $k>0$.
The map $f$ is also the limit of the family of algebraic maps of minimal degree $(3,-3k,3,0)$.
Indeed it is the limit of a family $f_t([x_0:x_1])=[(x_0^2+x_1^2)x_0,0,(x_0^2+tx_1^2)x_1,1]$
for $t>1$.
\qed
%%%%
}
%%%%%%
\end{exa}

%%%(Conjectures)%%
%%%%%%%%%%%%%%%%%%%%%
\paragraph{Conjectures concerning spaces of algebraic maps.}
%%%%%%%%%%%%%%%%%%%%%
The purpose of the third part of this section is to formally state the analogues 
of Theorem \ref{thm: I} and Corollary \ref{cor: II} concerning  approximation of spaces of continuous maps by algebraic maps. 
As explained in the introduction,  the analogous results are true in the complex case.%%%
%%%(Conjecture 8.7)%%%
\begin{conj}[cf. \cite{AKY1}, Conjecture 3.8]\label{conj: Alg}
%%%%%%%%%%%%%%%%%%%%%
Under the same assumptions as Theorem \ref{thm: I}, the natural projection maps
$$
\begin{cases}
\Psi^{\p}_{D}:A_D(m,\XS;g)\to \Alg_D^*(\RP^m,\XS;g),
\quad
%\\
\Psi_{D}:A_D(m,\XS)\to \Alg^*_D(\RP^m,\XS)
\\
\Gamma_{D}:\widetilde{A_D}(m,\XS)\to \Alg_D(\RP^m,\XS)
\end{cases}
$$
are homotopy equivalences.
\qed
\end{conj}
%%%%%%%%%
We strongly believe that Conjecture \ref{conj: Alg} is true.
As we have mentioned before, the natural projection $\Psi_D$ has contractible fibers.
If $\Psi_D$ is a quasi-fibration or satisfies the condition of the Vietoris-Begle theorem or has some other property of this kind,  then it must be a homotopy (or at least homology) equivalence.
We proved this for the simplest case in \cite{KY5}, but
the general case seems difficult as
the topology of the quotient space $\Alg^*_D(\RP^m,\XS)$
is complicated.
%%%%%%%%
%%
%%%%%(Conjecture 8.8)%%
%\begin{conj}[cf. \cite{AKY1}, Conjecture 3.9]\label{conj: quasi-fibration}
%%%%%%%%
%Under the same assumptions as Theorem \ref{thm: I}, the natural projection map
%$
%\Psi_{D}:A_D(m,\XS)\to \Alg^*_D(\RP^m,\XS)
%$
%is a quasi-fibration.
%\qed
%%%%%%
%\end{conj}
%%%%%%%%%%%

%%%()%%%%%%
%\par\vspace{4mm}\par
%\noindent{\bf Acknowledgements. }
%%%%%%%%%
%The authors should like to take this opportunity to thank  Jacob Mostovoy 
%for his valuable suggestions.

%\subsection{Subsection.}

\vspace{1cm}

%% \begin{figure}
%% \begin{center}
%% \includegraphics{figure.eps}
%% \caption{Caption}
%% \end{center}
%% \label{area}
%% \end{figure}

%% \begin{thebibliography}{9, 99 or Abc99}
%% \begin{thebibliography}{9}  for 1-digit labels
%% \begin{thebibliography}{99}  for 2-digit labels
%% \begin{thebibliography}{Abc}  for alphanumeric labels

\vspace{0.5cm}

\profile{Andrzej \textsc{Kozlowski}}
{%Organization /Affiliation name
Institute of Applied Mathematics and Mechanics,
University of Warsaw 
\\
Banacha 2, 02-097
Warsaw, Poland
\\
%E-mail:
akoz@mimuw.edu.pl}

\profile{Masahiro \textsc{Ohno}}
{%Organization /Affiliation name
Department of Mathematics,
University of Electro-Communications 
\\
%Address
Chofu, Tokyo 182-8585, Japan
\\
%E-mail: 
masahiro-ohno@uec.ac.jp
}

\profile{Kohhei \textsc{Yamaguchi}}
{%Organization /Affiliation name 
Department of Mathematics,
University of Electro-Communications
\\
Chofu, Tokyo 182-8585, Japan
\\
%E-mail: 
kohhei@im.uec.ac.jp}

\label{finishpage}

\end{document}